\documentclass[12pt]{amsart}
\usepackage{amsmath,amssymb,amsthm}
\usepackage{enumitem}
\usepackage{graphicx}
\usepackage{xcolor}
\newcommand\sbullet[1][.5]{\mathbin{\vcenter{\hbox{\scalebox{#1}{$\bullet$}}}}}
\theoremstyle{plain}
\newtheorem{theorem}{Theorem}[section]

\newtheorem{lemma}[theorem]{Lemma}
\theoremstyle{definition}
\newtheorem{definition}[theorem]{Definition}
\newtheorem{example}[theorem]{Example}
\newtheorem{remark}[theorem]{Remark}

\newtheorem*{problem}{Inverse Problem}
\makeatletter
\@addtoreset{equation}{section}
\makeatother

\makeatletter
\newcommand{\Spvek}[2][r]{%
  \gdef\@VORNE{1}
  \left(\hskip-\arraycolsep%
    \begin{array}{#1}\vekSp@lten{#2}\end{array}%
  \hskip-\arraycolsep\right)}

\def\vekSp@lten#1{\xvekSp@lten#1;vekL@stLine;}
\def\vekL@stLine{vekL@stLine}
\def\xvekSp@lten#1;{\def\temp{#1}%
  \ifx\temp\vekL@stLine
  \else
    \ifnum\@VORNE=1\gdef\@VORNE{0}
    \else\@arraycr\fi%
    #1%
    \expandafter\xvekSp@lten
  \fi}
\makeatother

\begin{document}
\title[Mixed Data in Inverse Spectral Problems]
{Mixed Data in Inverse Spectral Problems for the Schr\"{o}dinger Operators\\}
\author{Burak Hat\.{i}no\u{g}lu}
\address{Department of Mathematics, Texas A{\&}M University, College Station,
	TX 77843, U.S.A.}
\email{burakhatinoglu@math.tamu.edu}




\begin{abstract}
We consider the Schr\"{o}dinger operator on a finite interval with an $L^1$-potential. We prove that the potential can be uniquely 
recovered from one spectrum and subsets of another spectrum and point masses of the spectral measure (or norming constants) corresponding to 
the first spectrum. We also solve this Borg-Marchenko-type problem under some conditions on two spectra, when missing part of the second 
spectrum and known point masses of the spectral measure have different index sets.
\end{abstract}
\maketitle

\section{\bf {Introduction}}

In this paper, we consider the Schr\"{o}dinger (Sturm-Liouville) equation
\begin{equation*}
 Lu = -u''+qu = zu
\end{equation*}
on the interval $(0,\pi)$ with the boundary conditions 
\begin{align*}
  &u(0)\cos\alpha - u'(0)\sin\alpha = 0\\
 &u(\pi)\cos\beta + u'(\pi)\sin\beta = 0,
 \end{align*}
and a real-valued potential $q\in L^1(0,\pi)$. 
The spectrum $\sigma_{\alpha,\beta}$ of the Schr\"{o}dinger operator $L$ corresponding to these boundary conditions defines a discrete 
subset of the real line, bounded from below, diverging to $+\infty$.

Direct spectral problems aim to get spectral information from the potential. In inverse spectral problems, the goal is to recover the 
potential from spectral information, such as the spectrum, the norming constants, the spectral measure or Weyl-Titchmarsh $m$-function. 
These notions are discussed in Section 2.

The first inverse spectral result on Schr\"{o}dinger operators is given by Ambarzumian \cite{AMB}. He considered continuous potential with 
Neumann boundary conditions at both endpoints ($\alpha = \beta = \pi/2$) and showed that $q \equiv 0$ if the spectrum consists of squares 
of integers.

Later Borg \cite{BOR} proved that an $L^1$-potential is uniquely recovered from two spectra, corresponding to various pairs of 
boundary conditions and sharing the same boundary conditions at $\pi$ ($\beta_1 = \beta_2$), one of which should be Dirichlet boundary 
condition at $0$ ($\alpha_1 = 0$). Levinson \cite{LVS} extended Borg's result by removing the restriction of Dirichlet boundary 
condition at $0$.

Furthermore, Marchenko \cite{MAR} observed that the spectral measure (or Weyl-Titchmarsh $m$-function) uniquely recovers an $L^1$-potential.

Another classical result is due to Hochstadt and Lieberman \cite{HL}, which says that if the first half of an $L^1$-potential is known, one 
spectrum recovers the whole.

Statements of these classical results are given in Section 3.1.

Gesztesy, Simon and del Rio \cite{DGS} generalized Levinson's theorem to three spectra, by showing two thirds of the union of three spectra 
is sufficient spectral data to recover an $L^1$-potential.

Later on, Gesztesy and Simon \cite{GS} observed that extra smoothness conditions on the potential change required spectral data to recover the 
potential. They proved that the knowledge of the eigenvalues can be replaced by information on the derivatives of the potential. In addition, 
they \cite{GS} also generalized the Hochstadt-Lieberman theorem in the sense that more than the first half of an $L^1$-potential and a 
sufficiently large subset of a spectrum recover the potential.

Afterwards, Amour, Raoux and Faupin \cite{AFR,AR} proved similar results using extra information on the smoothness of the potential.

In a remarkable result, Horv\'{a}th \cite{HOR} characterized unique recovery of a potential in terms of completeness of an exponential system depending 
on given eigenvalues and known part of the potential. This observation opened a new path \cite{BBP,HOR,HS,MP} by connecting inverse 
spectral problems and completeness of exponential systems.

Moreover, Horv\'{a}th and S\'{a}f\'{a}r \cite{HS} proved similar results in terms of a cosine system. The cosine system depends on subsets 
of eigenvalues and norming constants and their spectral data consists of these two subsets.

Recently, Makarov and Poltoratski \cite{MP} gave a version of Horv\'{a}th's theorem (\cite{HOR}) in terms of exterior Beurling-Malliavin 
density by combining Horv\'{a}th's result and the Beurling-Malliavin theorem. In the same paper, they obtained another characterization 
result, which is an uncertainty version of Borg's theorem. As their spectral data, they considered a set of intervals known to include 
two spectra and characterized the inverse spectral problem in terms of a convergence criterion on this set of intervals.

All of these results mentioned above are discussed in Section 3.2.

Classical theorems of Borg, Levinson, Marchenko, Hochstadt and Lieberman led to various other inverse spectral results on 
Schr\"{o}dinger operators (see \cite{AF,GES,GUL,GUL2,GW,GW2,HOR2,PS,SAT,SIM,TU,WAN,WK,WSM,WY,WW,WW2,WX,WX2} and references therein). 
These problems can be divided into two groups. In Borg-Marchenko-type spectral problems, 
one tries to recover the potential from spectral data. However, Hochstadt-Lieberman-type (or mixed) spectral problems recover 
the potential using a mixture of partial information on the potential and spectral data.

In the present paper, our interest is on regular Schr\"{o}dinger operators with summable potentials on a finite interval. However, many 
problems with locally summable potentials \cite{ECK,ECK2,EGRT,ET,GS2,HM,KST} or on various settings such as half-line 
\cite{ET,GES,GS,GS3,SIM,TU}, real-line \cite{ET,GES,GS2,GS3,TU} or graphs \cite{BON,BON2,BS,BY,YAN} are solved.

Borg's, Levinson's and Hochstadt and Lieberman's theorems suggest that one spectrum gives exactly one half of the full spectral information required to recover the 
potential. Recalling the fact that the spectral measure is a discrete measure supported on a spectrum, the same can be said for the set of 
point masses of the spectral measure. As follows from Marchenko's theorem, the set of point masses of the spectral measure (or the set of 
norming constants) gives exactly one half of the full spectral information required to recover the potential.

These observations allow us to formulate the following question:
\begin{problem}
Do one spectrum and partial information on another spectrum and the set of point masses of the spectral measure corresponding to the 
first spectrum recover the potential?
\end{problem}
This Borg-Marchenko-type problem can be seen as a combination of Levinson's and Marchenko's results.

In the present paper, we answer this question positively. First, we give a proof with the most common boundary conditions, Dirichlet ($u=0$) 
and Neumann ($u'=0$). Theorem \ref{DDNDthm} solves this inverse spectral problem when given part of the point masses of the spectral measure 
corresponding to the Dirichlet-Dirichlet spectrum matches with the missing part of the Neumann-Dirichlet spectrum, i.e. they share same 
index sets. In Theorem \ref{DDNDindex1} and Theorem \ref{DDNDindex2}, we consider the non-matching index sets case with some restrictions 
on two spectra.

In order to deal with general boundary conditions we introduce a more general $m$-function in Section 4.3. With this $m$-function, we extend 
Theorem \ref{DDNDthm} in Theorem \ref{GBCthm} to general boundary conditions. 
In Theorem \ref{GBCindex1} and Theorem \ref{GBCindex2} we consider the non-matching index sets case.\\

The paper is organized as follows.\\
\begin{itemize}
 \item In Section 2.1 we discuss spectra of Schr\"{o}dinger operators and their asymptotics for various boundary conditions.
 \item In Section 2.2 we define Weyl-Titchmarsh $m$-function and spectral measure for Schr\"{o}dinger operators.
 \item In Section 3.1 we recall statements of the classical results of Ambarzumian, Borg, Levinson, Marchenko, Hochstadt and Lieberman.
 \item In Section 3.2 we discuss some recent results in the finite interval setting with summable potential. 
 \item In Section 4.1 we give a representation of Weyl-Titchmarsh $m$-function as an infinite product and prove the inverse spectral 
 problem mentioned above with Dirichlet-Dirichlet, Neumann-Dirichlet boundary conditions.
 \item In Section 4.2 we consider the same problem in the non-matching index sets case.
 \item In Section 4.3 we introduce a more general $m$-function and solve the inverse spectral problem corresponding to this $m$-function 
 with general boundary conditions in both the matching and non-matching index sets cases.
 \item In Appendix A we list all definitions and theorems from complex function theory used in this paper.
\end{itemize}

\newpage

 \section{\bf {Preliminaries}}
 
 \subsection{\bf {One-dimensional Schr\"{o}dinger operator on a finite interval}}
 
 As it was defined in the introduction, we consider the Schr\"{o}dinger equation
 \begin{equation}\label{Scheq}
   Lu = -u'' + qu = zu
 \end{equation}
on the interval $(0,\pi)$ associated with the boundary conditions
 \begin{align}
  &u(0)\cos\alpha - u'(0)\sin\alpha = 0 \label{BC1}\\
 &u(\pi)\cos\beta + u'(\pi)\sin\beta = 0, \label{BC2}
 \end{align}
 where $\alpha, \beta \in [0,\pi)$ and the potential $q \in L^1(0,\pi)$ is real-valued.
 
 The spectrum $\sigma_{\alpha,\beta}$ of the Schr\"{o}dinger operator 
 \begin{equation*}
  L : u \mapsto -u'' + qu
 \end{equation*}
 with $q \in L^1$ and boundary conditions (\ref{BC1}), (\ref{BC2}) is a discrete real sequence, bounded from below. Adding a positive constant to 
 the potential $q$, shifts the spectrum by the same constant. This allows us to assume wlog 
 $\sigma_{\alpha,\beta} \subset \mathbb{R}_+$. Throughout the paper we assume $\mathbb{N} = \{1,2,3,\dots\}$.
 Asymptotic behavior of the spectrum $\sigma_{\alpha,\beta} = \{a_n\}_{n\in \mathbb{N}}$, 
 depending on the signs of $\alpha$ and $\beta$, is given below:
 
 If $\alpha \neq 0$, $\beta \neq 0$, then 
 \begin{equation}\label{asy11}
  a_n = (n-1)^2 + \frac{2}{\pi}[\cot(\beta) + \cot(\alpha)] + \frac{1}{\pi}\int_0^{\pi}q(x)dx + \alpha_n
 \end{equation}
where $\alpha_n = o(1)$ as $n \to + \infty$.

If $\alpha = 0$, $\beta = 0$, then 
 \begin{equation}\label{asy00}
  a_n = n^2 + \frac{1}{\pi}\int_0^{\pi}q(x)dx + \alpha_n
 \end{equation}
where $\alpha_n = o(1)$ as $n \to + \infty$.

If $\alpha \neq 0$, $\beta = 0$, then 
 \begin{equation}\label{asy10}
  a_n = \left(n-\frac{1}{2}\right)^2 + \frac{2}{\pi}\cot(\alpha) + \frac{1}{\pi}\int_0^{\pi}q(x)dx + \alpha_n
 \end{equation}
where $\alpha_n = o(1)$ as $n \to + \infty$.

If $\alpha = 0$, $\beta \neq 0$, then 
 \begin{equation}\label{asy01}
  a_n = \left(n-\frac{1}{2}\right)^2 + \frac{2}{\pi}\cot(\beta) + \frac{1}{\pi}\int_0^{\pi}q(x)dx + \alpha_n
 \end{equation}
where $\alpha_n = o(1)$ as $n \to + \infty$.

In the case $q \in L^2(0,\pi)$, the same asymptotics are valid with $\{\alpha_n\}_{n \in \mathbb{N}} \in l^2$.

One can find these results in the classical texts on Schr\"{o}dinger operators, for instance 
\cite{LG} or \cite{LS}.

\subsection{\bf {Weyl-Titchmarsh $m$-function and the spectral measure}}
 
 Let us choose the boundary condition (\ref{BC1}) and introduce two solutions $s_z(t)$ and $c_z(t)$ of (\ref{Scheq}) satisfying the initial 
 conditions 
\begin{align*}
&s_z(0) = \sin(\alpha), \quad s'_z(0) = \cos(\alpha)\\
&c_z(0) = \cos(\alpha), \quad c'_z(0) = -\sin(\alpha).
\end{align*}
\begin{definition}
 The norming constant $\tau_{\alpha}$, for the eigenvalue $a_n$ is defined as
 \begin{equation*}
  \tau_{\alpha}(a_n) := \int_0^{\pi}|s_{a_n}(t)|^2dt.
 \end{equation*}
\end{definition}
Note that $s_z(t)$ and $c_z(t)$ are linearly independent solutions and 
their Wronskian satisfies $W(c_z,s_z) = 1$. This allows us to represent $u_z(t)$, a solution of ($\ref{Scheq}$) with boundary 
conditions $u_z(\pi) = \sin\beta$, $u_z'(\pi) = -\cos\beta$, as 
\begin{equation*}
u_z(t) = c_z(t) + m_{\alpha,\beta}(z)s_z(t),
\end{equation*}
where 
\begin{equation*}
    m_{\alpha,\beta}(z) = -\frac{W(c_z,u_z)}{W(s_z,u_z)}.
\end{equation*}
This is how we derive the $m$-function.
\begin{definition}
 Weyl-Titchmarsh $m$-function with the boundary conditions (\ref{BC1}), (\ref{BC2}) is defined as
 \begin{equation*}
  m_{\alpha,\beta}(z) := \frac{\cos(\alpha)u_z'(0) + \sin(\alpha)u_z(0)}{-\sin(\alpha)u_z'(0) + \cos(\alpha)u_z(0)},
 \end{equation*}
where $\alpha,\beta \in [0,\pi)$.
\end{definition}
It is well-known that Weyl $m$-function  $m_{\alpha,\beta}$ is a meromorphic Herglotz function. 
The definition of a Herglotz function and other definitions and results from complex function theory used in this paper can be found in 
Appendix A. Everitt \cite{EVE} proved that the Weyl $m$-function has the asymptotic
\begin{equation*}
 m_{0,\beta}(z) = i\sqrt z + o(1)
\end{equation*}
for $\alpha = 0$, and
\begin{equation*}
 m_{\alpha,\beta}(z) = \frac{\cos\alpha}{\sin\alpha} + \frac{1}{\sin^2\alpha}\frac{i}{\sqrt z} + O\left(\frac{1}{|z|}\right)
\end{equation*}
for $\alpha \in (0,\pi)$ as $z$ goes to infinity in the upper half plane. Asymptotics of Weyl $m$-function and Herglotz 
representation theorem imply that $m_{\alpha,\beta}$ is represented as the Herglotz integral of a discrete positive Poisson-finite 
measure supported on the spectrum $\sigma_{\alpha,\beta}$:
\begin{equation}\label{Hrep}
 m_{\alpha,\beta}(z) = a + \int_{\mathbb{R}} \left[\frac{1}{t-z} - \frac{t}{1+t^2}\right]d\mu_{\alpha,\beta}(t), 
\end{equation}
where $a = \Re( m_{\alpha,\beta}(i))$, $\sigma_{\alpha,\beta} = \{a_n\}_{n \in \mathbb{N}}$ and 
$\mu_{\alpha,\beta} = \sum_{n \in \mathbb{N}} \gamma_n \delta_{a_n}$. The measure $\mu_{\alpha,\beta}$ is the spectral 
measure of the Schr\"{o}dinger operator $L$ corresponding to the $m$-function $m_{\alpha,\beta}$. The point masses of the spectral measure 
is represented in terms of norming constants as $\gamma_n = (\tau_{\alpha}(a_n))^{-1}.$ 
\begin{definition}
The spectral measure of the Schr\"{o}dinger operator $L$ corresponding to the $m$-function $m_{\alpha,\beta}$ (or the boundary 
conditions (\ref{BC1}), (\ref{BC2})) is defined as
 \begin{equation*}
  \mu_{\alpha,\beta} := \sum_{n \in \mathbb{N}} \frac{\delta_{a_n}}{\tau_{\alpha}(a_n)},
 \end{equation*}
where $\alpha,\beta \in [0,\pi)$ and $\sigma_{\alpha,\beta} = \{a_n\}_{n \in \mathbb{N}}$.
\end{definition}
Since $\mu_{\alpha,\beta}$ is a Poisson-finite 
measure, the spectrum and the point masses of the spectral measure satisfy 
\begin{equation*}
\sum_{n \in \mathbb{N}} \frac{\gamma_n}{1+a_n^2} ~\textless~ \infty.
\end{equation*}

These properties of the $m$-function, the spectral measure and a detailed 
discussion of one dimensional Schr\"{o}dinger operators appear in Chapter 9 of \cite{TES}.

In order to illustrate what we have discussed so far, let us consider the free potential ($q \equiv 0$) with Dirichlet ($u=0$) and 
Neumann ($u'=0$) boundary conditions.

\begin{example}\label{exmp}
 The spectra, the $m$-function and the spectral measure for $q \equiv 0$ on $(0,\pi)$ with Dirichlet-Dirichlet, Neumann-Dirichlet and 
 Neumann-Neumann boundary conditions are as follows.
\begin{equation*}
\begin{aligned}[c]
&\sigma_{DD} := \sigma_{0,0} = \{n^2\}_{n \in \mathbb{N}}\\
&\sigma_{ND} := \sigma_{\pi/2,0} = \{(n-\frac{1}{2})^2\}_{n \in \mathbb{N}}\\
&\sigma_{NN} := \sigma_{\pi/2,\pi/2} = \{(n-1)^2\}_{n \in \mathbb{N}}\\
\end{aligned}
\quad 
\begin{aligned}[c]
&m_{0,0} = -\sqrt{z}\cot(\sqrt{z}\pi)\\
&m_{\pi/2,0} = \frac{\tan(\sqrt{z}\pi)}{\sqrt{z}}\\
&m_{\pi/2,\pi/2} = \frac{\cot(\sqrt{z}\pi)}{\sqrt{z}}\\
\end{aligned}
\quad
\begin{aligned}[c]
&\mu_{0,0} = \frac{2}{\pi}\sum_{n=1}^{\infty}n^2\delta_{n^2}\\
&\mu_{\pi/2,0} = \frac{2}{\pi}\sum_{n=1}^{\infty}\delta_{(n-1/2)^2}\\
&\mu_{\pi/2,\pi/2} = \frac{2}{\pi}\sum_{n=1}^{\infty}\delta_{(n-1)^2}\\
\end{aligned}
\end{equation*}
\end{example}

\begin{figure}[h]
    \centering
    \includegraphics[scale=0.45]{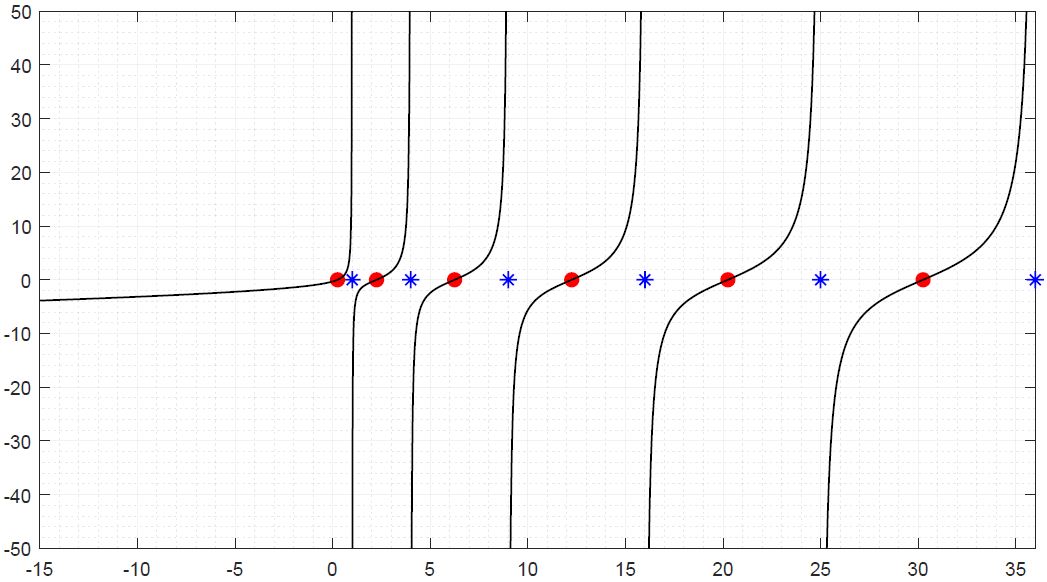}
    \caption{The graph of Weyl $m$-function $m_{0,0}$ on $\mathbb{R}$, Neumann-Dirichlet spectrum $\sigma_{ND}$ 
    ($\textcolor{red}{\sbullet[0.75]}$) and Dirichlet-Dirichlet spectrum $\sigma_{DD}$ (\textcolor{blue}{$\ast$}) 
    for the free potential ($q \equiv 0$).}
\end{figure}

 \section{\bf {Inverse spectral theory of regular Schr\"{o}dinger operators}}
 \subsection{Classical results}
 The first inverse spectral result on Schr\"{o}dinger operators was given by Ambarzumian.
 \begin{theorem} [{Ambarzumian \cite{AMB},~\cite{HOR}}] Let $q \in C[0,\pi]$ and $\sigma_{\pi/2,\pi/2} = \{n^2\}_{n=0}^{\infty}$.
 Then $q \equiv 0$.  
 \end{theorem}
 Later Borg found that in most cases two spectra is the required spectral information to recover the operator uniquely.
 \begin{theorem} [{Borg \cite{BOR},~\cite{HOR}}]\label{Borg} Let $q \in L^1(0,\pi)$, $\sigma_1 = \sigma_{0,\beta}$, 
 $\sigma_2 = \sigma_{\alpha_2,\beta}$, $\sin\alpha_2 \neq 0$ and\\
 $\widetilde{\sigma}_2$ = 
    $\begin{cases}
       \sigma_2 & \text{if }\sin\beta = 0\\
       \sigma_2 \text{\textbackslash} \{a_1\} &\text{if }\sin\beta \neq 0.  
     \end{cases}$\\
 Then $\sigma_1\cup\widetilde{\sigma}_2$ determines the potential and no proper subset has the same property.
 \end{theorem}
A Schr\"{o}dinger operator (or a potential) is said to be determined (or recovered) by its spectral data, if any other operator with the same 
data must have the same potential a.e.\ on $(0,\pi)$. Levinson extended Borg's result by removing the Dirichlet boundary condition restriction 
from the first spectrum.
\begin{theorem} [{Levinson \cite{LVS},~\cite{HOR}}]\label{Levinson} Let $q \in L^1(0,\pi)$ and $\sin(\alpha_1-\alpha_2) \neq 0$. Then 
$\sigma_{\alpha_1,\beta}$ and $\sigma_{\alpha_2,\beta}$ determine the potential.
\end{theorem}

Marchenko showed that the spectral measure or the corresponding Weyl $m$-function provides sufficient spectral data to recover the potential 
uniquely.
\begin{theorem} [{Marchenko \cite{MAR},~\cite{TES}-Section 9.4}] Let $q \in L^1(0,\pi)$. Then $\mu_{\alpha,\beta}$ or 
$m_{\alpha,\beta}$ determines the potential.
\end{theorem}
In the notations of Section 2.2, Marchenko's theorem says that the spectrum $\sigma_{\alpha,\beta} = \{a_n\}_{n \in \mathbb{N}}$ and the 
point masses $\{\gamma_n\}_{n \in \mathbb{N}}$ of the corresponding spectral measure (or the norming constants 
$\{\tau_{\alpha}(a_n)\}_{n \in \mathbb{N}}$) provide sufficient spectral data to recover the 
operator uniquely.

Hochstadt and Lieberman observed that one spectrum recovers the potential if the first half of it is known.
\begin{theorem} [{Hochstadt, Lieberman \cite{HL}}] Let $q \in L^1(0,\pi)$. Then $q$ on $(0,\pi/2)$ and 
$\sigma_{\alpha,\beta}$ determine the potential.
\end{theorem}

\subsection{Some recent results in the finite interval case}

For any discrete real sequence $A = \{x_n\}_{n\in \mathbb{N}}$, $x_n \rightarrow \infty$ the counting function is defined as 
 \begin{equation*}
     n_A(t) := \sum_{x_n\leq t} 1.
 \end{equation*}
 
Gesztesy, Simon and del Rio generalized Levinson's theorem to three spectra.
\begin{theorem} [{del Rio, Gesztesy, Simon \cite{DGS}}] Let $q \in L^1(0,\pi)$. Then 
$S \subset \sigma_{\alpha_1,\beta}\cup\sigma_{\alpha_2,\beta}\cup\sigma_{\alpha_3,\beta}$ satisfying 
 \begin{equation*}
  n_S(t) \geq (2/3)n_{(\sigma_{\alpha_1,\beta}\cup\sigma_{\alpha_2,\beta}\cup\sigma_{\alpha_3,\beta})}(t)
 \end{equation*}
 for sufficiently large $t ~\textgreater~ 0$, determine the potential.
\end{theorem}

Gesztesy and Simon observed that the knowledge of the eigenvalues can be replaced by information on the derivatives of the potential around 
the midpoint of the interval.
    \begin{theorem}[{Gesztesy, Simon \cite{GS2}}]
    Let $q \in L^1(0,\pi)$, $\alpha , \beta \neq 0$ and $q \in C^{2k}(\pi/2-\epsilon, \pi/2 +\epsilon)$ for some $k \in \mathbb{N}$ and 
    $\epsilon ~\textgreater~ 0$. Then $q$ on $(0,\pi/2)$ and $\sigma_{\alpha,\beta}$ except for $k+1$ eigenvalues determine the potential.
    \end{theorem}
    In the same paper, they generalized Hochstadt-Lieberman theorem.
 \begin{theorem} [{Gesztesy, Simon \cite{GS2}}] Let $q \in L^1(0,\pi)$ and $\pi/2 ~\textless~ a ~\textless~ \pi$. Then $q$ on $(0,a)$ and
 $S \subset \sigma_{\alpha,\beta}$ satisfying 
 \begin{equation*}
  n_S(t) \geq 2(1-a/\pi)n_{\sigma_{\alpha,\beta}}(t) + a/\pi - 1/2
 \end{equation*}
 for sufficiently large $t ~\textgreater~ 0$, determine the potential.
\end{theorem}

Amour, Raoux and Faupin proved similar results using extra information on smoothness of the potential.
 \begin{theorem} [{Amour, Raoux \cite{AR}}] Let $\alpha, \beta_1,\beta_2 \neq 0$, $p\in [1,\infty)$, $q_1,q_2 \in L^1(0,\pi)$, 
 $q_1-q_2\in L^p(a,\pi)$ and $\pi/2 ~\textless~ a ~\textless~ \pi$. If $q_1=q_2$ a.e. on $(0,a)$ and
 $S \subset \sigma_{\alpha,\beta_1}(q_1) \cap \sigma_{\alpha,\beta_2}(q_2)$ satisfies 
 \begin{equation*}
  2(1-a/\pi)n_{\sigma}(t) + C \geq n_S(t) \geq  2(1-a/\pi)n_{\sigma}(t) + 1/(2p) + 2a/\pi - 2
 \end{equation*}
 for a real number $C$ and sufficiently large $t ~\textgreater~ 0$, where $\sigma$ denotes either of $\sigma_{\alpha,\beta_k}(q_k)$, 
 then $q_1=q_2$ a.e. on $(0,\pi)$.
\end{theorem}

\begin{theorem}[{Amour, Faupin, Raoux \cite{AFR}}]
    Let $\alpha, \beta_1, \beta_2 \neq 0$, $k\in \{0,1,2\}$, $p\in [1,\infty)$, $q_1,q_2 \in W^{k,1}(0,\pi)$, $q_1-q_2\in W^{k,p}(a,\pi)$ 
    and $\pi/2 ~\textless~ a ~\textless~ \pi$. If $q_1=q_2$ on $(0,a)$ and 
    $S \subset \sigma_{\alpha,\beta_1}(q_1) \cap \sigma_{\alpha,\beta_2}(q_2)$ satisfying 
 \begin{equation*}
  n_S(t) \geq  2(1-a/\pi)n_{\sigma}(t) - k/2 + 1/(2p) + a/\pi - 3/2
 \end{equation*}
 for sufficiently large $t ~\textgreater~ 0$, where $\sigma$ denotes either of $\sigma_{\alpha,\beta_k}(q_k)$, 
 then $q_1=q_2$ a.e. on $(0,\pi)$.
    \end{theorem}

    \begin{theorem}[{Amour, Faupin, Raoux \cite{AFR}}]
    Let $\alpha, \beta_1, \beta_2 \neq 0$, $k\in \{0,1,2\}$, $p\in [1,\infty)$, $q_1,q_2 \in W^{k,1}(0,\pi)$, $q_1-q_2\in W^{k,p}(a,\pi)$ 
    and $\pi/2 ~\textless~ a ~\textless~ \pi$. If $q_1=q_2$ on $(0,a)$ and 
    $S \subset \sigma_{\alpha,\beta_1}(q_1) \cap \sigma_{\alpha,\beta_2}(q_2)$ satisfying 
 \begin{equation*}
  2(1-a/\pi)n_{\sigma}(t) + C \geq n_S(t) \geq  2(1-a/\pi)n_{\sigma}(t) - k/2 + 1/(2p) + 2a/\pi - 2
 \end{equation*}
 for sufficiently large $t ~\textgreater~ 0$, where $\sigma$ denotes either of $\sigma_{\alpha,\beta_k}(q_k)$, 
 then $q_1=q_2$ a.e. on $(0,\pi)$.
    \end{theorem}
    
Horv\'{a}th proved a remarkable characterization theorem, which represents a connection between inverse spectral theory and completeness 
of exponential systems.
\begin{theorem}[{Horv\'{a}th \cite{HOR}}]\label{HORV} Let $1\leq p \leq \infty$, $q \in L^p(0,\pi)$, $0 \leq a ~\textless~ \pi$ and 
$\lambda_n \in \sigma_{\alpha,0}$. Then $q$ on $(0,a)$ and the eigenvalues $\lambda_n$ determine $q$ if and only if the system 
 \begin{equation*}
  e(\Lambda) = \{ e^{\pm2i\mu x},e^{\pm2i\sqrt{\lambda_n}x}: \quad n\geq 1 \}
 \end{equation*}
 is complete in $L^p(a-\pi,\pi-a)$ for some $\mu \neq \pm\sqrt{\lambda_n}$.
\end{theorem}

Horv\'{a}th and S\'{a}f\'{a}r proved similar results for the norming constants in terms of a cosine system. For a sequence 
$\Lambda = \{\lambda_1,\lambda_2,\dots\}\subset \mathbb{R}$ and a subset $S\subset \Lambda$ they considered the cosine system:
\begin{equation*}
    C(\Lambda,S) = \{\cos(2\sqrt{\lambda_n}t):n\in \mathbb{N}\} \cup \{t\cos(2\sqrt{\lambda_n}t):\lambda_n \in S\}.
\end{equation*}

\begin{theorem}[{Horv\'{a}th, S\'{a}f\'{a}r \cite{HS}}] \label{HorSaf} Let $\beta \neq 0$, $1\leq p \leq \infty$, 
$q \in L^1(0,\pi)$, $q \in L^p(a,\pi)$, $0 \leq a ~\textless~ \pi$ and 
\begin{equation*}
    \Lambda = \{\lambda_n: \lambda_n \in \sigma_{\alpha_n,\beta}, n\in \mathbb{N}\}
\end{equation*}
be a subset of eigenvalues such that $\lambda_n \not\to -\infty$ are different real numbers and $S\subset \Lambda$.
Then $q$ on $(0,a)$, $\Lambda$ and $\{\tau_{\alpha_n}(\lambda_n)\}_{\lambda_n \in S}$ determine $q$ if the system $C(\Lambda, S)$ is complete in $L^p(0, \pi-a)$.
\end{theorem}

For Dirichlet boundary condition Horv\'{a}th and S\'{a}f\'{a}r obtained an optimal condition.
\begin{theorem} [{Horv\'{a}th, S\'{a}f\'{a}r \cite{HS}}]
Let us have the assumptions of Theorem \ref{HorSaf}, but $\beta = 0$. Let $\mu \neq \pm \sqrt{\lambda_n}$, $\mu \in \mathbb{R}$. Then the 
system $C(\Lambda,S) \cup \{\cos(2\sqrt{\mu}t)\}$ is complete in $L^p(0,\pi-a)$ if and only if $q$ on $(0,a)$, $\Lambda$ and 
$\{\tau_{\alpha_n}(\lambda_n)\}_{\lambda_n \in S}$ determine $q$.
\end{theorem}

Makarov and Poltoratski gave a characterization theorem in terms of exterior Beurling-Malliavin density as a corollary of Horv\'{a}th's 
result \cite{HOR} (Theorem \ref{HORV} above) and the Beurling-Malliavin theorem \cite{BM,BM2}.

If $\{I_n\}_{n\in \mathbb{N}}$ is a sequence of disjoint intervals on the real line, it is called short if 
\begin{equation*}
 \sum_{n\in \mathbb{N}}\frac{|I_n|^2}{1+dist^2(0,I_n)} ~\textless~ \infty
\end{equation*}
and long otherwise.

If $\Lambda$ is a sequence of real points, its exterior (effective) Beurling-Malliavin density is defined as 
\begin{equation*}
 D^*(\Lambda) = \sup \{d \text{ } | \text{ } \exists \text{ long } \{I_n\} 
 \text{ such that } \#(\Lambda \cap I_n) \geq d|I_n|, \text{ } \forall n\in \mathbb{N}\}.
\end{equation*}
For a non-real sequence its density is defined as $D^*(\Lambda) = D^*(\Lambda')$, where $\Lambda'$ is a real sequence 
$\lambda'_n = (\Re\frac{1}{\lambda_n})^{-1}$, if $\Lambda$ has no imaginary points, and as $D^*(\Lambda) = D^*((\Lambda+c)')$ otherwise.\\
For any complex sequence $\Lambda$ its radius of completeness is defined as 
\begin{equation*}
 R(\Lambda) = \sup\{ a ~|~ \{e^{i\lambda z}\}_{\lambda \in \Lambda} \text{ is complete in } L^2(0,a)\}.
\end{equation*}
Now we are ready to state one of the fundamental results of Harmonic Analysis.

\begin{theorem} [Beurling-Malliavin theorem \cite{BM,BM2}] Let $\Lambda$ be a discrete sequence. Then 
\begin{equation*}
 R(\Lambda) = 2\pi D^*(\Lambda).
\end{equation*}
\end{theorem}

Let us note that Makarov and Poltoratski considered the Schr\"{o}dinger equation $Lu = -u''+qu = z^2u$ and the $m$-function corresponding to 
this equation, which is obtained by applying the square root transform to the $m$-function we have discussed so far. Let us denote their 
$m$-function by $\widetilde{m}$.

\begin{theorem}[Makarov, Poltoratski \cite{MP}]
 Let $\Lambda = \{\lambda_n\}_{n \in \mathbb{N}}$ be a sequence of discrete non-zero complex numbers, $q\in L^2(0,\pi)$ and $0 \leq a \leq 1$. 
 The following statements are equivalent:
 \begin{enumerate}
  \item $q$ on $(0,d)$ for some $d ~\textgreater~ a$ and $\{\widetilde{m}(\lambda_n)\}_{n \in \mathbb{N}}$ determine $q$.
  \item $\pi D^*(\Lambda) \geq 1-a$.
 \end{enumerate}
\end{theorem}

Makarov and Poltoratski's observation shows that Horv\'{a}th's theorem establishes equivalence between mixed spectral problems for 
Schr\"{o}dinger operators and the Beurling–Malliavin problem on completeness of exponentials in $L^2$ spaces.

In the same paper they obtained an uncertainty version of Borg's theorem.
\begin{theorem}[Makarov, Poltoratski \cite{MP}]
 Let $\{I_n\}_{n \in \mathbb{N}}$ be a sequence of intervals on $\mathbb{R}$ and $q\in L^2(0,\pi)$. 
 The following statements are equivalent:
 \begin{enumerate}
  \item The condition $\sigma_{DD} \cup \sigma_{ND} \subset \cup_{n \in \mathbb{N}}I_n$ and $q$ on $(0,\epsilon)$ for some 
  $\epsilon ~\textgreater~ 0$ determine the potential $q$.
  \item For any long sequence of intervals $\{J_n\}_{n \in \mathbb{N}}$, 
  \begin{equation*}
   \frac{\sum_{I_n \cap J_n} \log_{-}|I_n|}{|J_n|} \nrightarrow 0
  \end{equation*}
  as $n \rightarrow \infty$.

 \end{enumerate}
\end{theorem}

\section{\bf {An inverse spectral problem with mixed data}}
 
  \subsection{\bf{The main result with Dirichlet-Dirichlet and Neumann-Dirichlet boundary conditions}}

  We prove our main result, Theorem \ref{DDNDthm}, by representing the Weyl-Titchmarsh $m$-function as an infinite product
  in terms of Dirichlet-Dirichlet ($\alpha = 0$, $\beta = 0$) and Neumann-Dirichlet ($\alpha = \pi/2$, $\beta = 0$) spectra. 
  We follow the notations introduced in Example \ref{exmp} for these two spectra, i.e. $\sigma_{DD} := \sigma_{0,0}$ and $\sigma_{ND} := \sigma_{\pi/2,0}$. 
  For simplicity, let us also denote $m_{0,0}$ by $m$. For any infinite product (or sum) defined on an open subset 
  $\Omega \subset \mathbb{C}$, normal convergence means that the product (or the sum) converges uniformly on every compact subset 
  of $\Omega$.
 
 \begin{lemma}\label{lmm1}
  The $m$-function of a regular Schr\"{o}dinger operator $(q\in L^1(0,\pi))$ for Dirichlet-Dirichlet boundary conditions ($\alpha = 0$, 
  $\beta = 0$) has representations in terms of Dirichlet-Dirichlet and Neumann-Dirichlet spectra:
  \begin{equation}\label{mfcn1}
   m(z) = C \left(\frac{z}{b_1}-1\right) \prod_{n\in \mathbb{N}}\left(\frac{z}{b_{n+1}}-1\right)\left(\frac{z}{a_n}-1\right)^{-1},
  \end{equation}
  and
 \begin{equation}\label{mfcn2}
   m(z) = -C \prod_{n\in \mathbb{N}}\left(\frac{z}{b_{n}}-1\right)\left(\frac{z}{a_n}-1\right)^{-1},
  \end{equation}
  where $C ~\textgreater~ 0$, $\sigma_{DD} = \{a_n\}_{n \in \mathbb{N}}$, $\sigma_{ND} = \{b_n\}_{n \in \mathbb{N}}$ and 
  the product converges normally on $\displaystyle \mathbb{C} \text{\textbackslash} \cup_{n \in \mathbb{N}} a_n$.
 \end{lemma}
 
 \begin{proof}
  Let $m = u'_{z}(0)/u_{z}(0)$ be the Weyl $m$-function 
with boundary conditions $u(\pi)=0$, $u'(\pi)=-1$. Since $m$ is a meromorphic Herglotz function, $\Theta:=\frac{m-i}{m+i}$ is the 
corresponding meromorphic inner function. See Appendix A for the definition of a meromorphic inner function and the relation between 
Herglotz and inner functions.

Let us define the set $E$ in $\mathbb{R}$ as $E:=\{z\in \mathbb{R}: Im\Theta ~\textgreater~ 0\}$.
The set $E$ is given in terms of $\sigma_{DD}=\{a_n\}_{n\in\mathbb{N}}$ and $\sigma_{ND}=\{b_n\}_{n\in\mathbb{N}}$, namely

\begin{equation*}
 E = (-\infty , b_1) \cup \cup_{n \in \mathbb{N}} (a_{n} , b_{n+1}).
\end{equation*}

The characteristic function of $E$ coincides with the real part of the function 
$\frac{1}{i\pi}\log(i \frac{1+\Theta}{1-\Theta})$ a.e. on $\mathbb{R}$. Since $m$ is a meromorphic Herglotz function mapping $\mathbb{R}$
to $\mathbb{R}$ a.e., $\log(m)=\log(i \frac{1+\Theta}{1-\Theta})$ is a well-defined holomorphic function on $\mathbb{C}_+$
and its imaginary part takes values $0$ and $\pi$ on $\mathbb{R}$. Therefore 
$\frac{1}{i\pi}\log(m)=\frac{1}{i\pi}\log(i \frac{1+\Theta}{1-\Theta})$ and the Schwarz integral of $\chi_E$, $S_{\chi_E}$ 
differ by a purely imaginary number on a.e. $\mathbb{R}$, i.e.

\begin{equation*}
\frac{1}{i\pi}\log\left(i \frac{1+\Theta}{1-\Theta}\right) = S_{\chi_E}+ic = P_{\chi_E}+iQ_{\chi_E}+ic, \quad c\in\mathbb{R},
\end{equation*}
where $P$ and $Q$ are Poisson and conjugate Poisson integrals of $\chi_E$, respectively. Definitions of $S$, $P$ and $Q$ appear in 
the appendix.
Therefore
\begin{equation*}
i \frac{1+\Theta}{1-\Theta} = \exp(i\pi S_{\chi_E} - \pi c) = \exp(i\pi P_{\chi_E} - \pi Q_{\chi_E} - \pi c),\quad c\in\mathbb{R}.
\end{equation*}

On the real line, $\exp(S_h) = \exp(h+i\widetilde{h})$ for any Poisson-summable function $h$, where $\widetilde{h}$ is the Hilbert transform
of $h$. If we let $h:=\chi_E$, then 

\begin{equation*}
\widetilde{h}(x) =  \frac{1}{\pi} \left[\log\left(\frac{\sqrt{1+b^2_1}}{|x-b_1|}\right) 
+ \sum_{n\in\mathbb{N}} \log\left(\frac{|x-a_n|}{|x-b_{n+1}|}\right) 
+ \frac{1}{2} \sum_{n\in\mathbb{N}} \log\left(\frac{1+b^2_{n+1}}{1+a^2_n}\right)\right].
\end{equation*}
Therefore
\begin{equation*}
\exp(-\pi\widetilde{h}(x)) = \frac{|x-b_1|}{\sqrt{1+b^2_1}} \prod_{n\in\mathbb{N}} \frac{|x-b_{n+1}|}{|x-a_n|} 
 \prod_{n\in\mathbb{N}}\left(\frac{1+a^2_n}{1+b^2_{n+1}}\right)^{1/2}.
\end{equation*}
Noting that $\exp(i\pi h)$ is $-1$ on $E$ and $1$ on $\mathbb{R}\text{\textbackslash}E$, the Weyl $m$-function can be given 
in terms of $\sigma_{DD}$ and $\sigma_{ND}$ a.e. on $\mathbb{R}$:
\begin{align*}
m(x) &= i \frac{1+\Theta(x)}{1-\Theta(x)}\\ 
&= \exp(i\pi S_{\chi_E} - \pi c)\\ 
&= \frac{x-b_1}{\sqrt{1+b^2_1}} \prod_{n\in\mathbb{N}} \frac{x-b_{n+1}}{x-a_n} \prod_{n\in\mathbb{N}}\left(\frac{1+a^2_n}{1+b^2_{n+1}}\right)^{1/2}\exp(-\pi c)\\
&= C \left(\frac{x}{b_1}-1\right) \prod_{n\in \mathbb{N}}\left(\frac{x}{b_{n+1}}-1\right)\left(\frac{x}{a_n}-1\right)^{-1}\\
\end{align*}
where $C = \exp(-\pi c)\prod_{n\in\mathbb{N}}\frac{\sqrt{1+a^2_n}}{a_n}\frac{b_n}{\sqrt{1+b^2_n}}$.\\ 
Since $m(z)$ and $C \left(\frac{z}{b_1}-1\right) \prod_{n\in \mathbb{N}}\left(\frac{z}{b_{n+1}}-1\right)\left(\frac{z}{a_n}-1\right)^{-1}$ 
are meromorphic functions that agree a.e. on $\mathbb{R}$, they are identical by the identity theorem for meromorphic functions.
This gives the first representation (\ref{mfcn1}). The second representation (\ref{mfcn2}) follows from normal convergence of 
$ \{z/b_n-1\}_{n\in \mathbb{N}}$ to $-1$ in $\mathbb{C}$.
 
 \end{proof}

Using this representation of the $m$-function, we prove our main result. 
At this point let us note that the points in a spectrum are enumerated in increasing order, which is done following the asymptotics 
(\ref{asy11}), (\ref{asy00}), (\ref{asy10}) and (\ref{asy01}).

\begin{theorem}\label{DDNDthm} Let $q \in L^1(0,\pi)$ and $A \subseteq \mathbb{N}$. Then $\{a_n\}_{n \in \mathbb{N}}$, 
$\{b_n\}_{n \in \mathbb{N}\text{\textbackslash} A}$ and $\{\gamma_n\}_{n \in A}$ determine the potential $q$, 
where $\sigma_{DD} = \{a_n\}_{n \in \mathbb{N}}$, $\sigma_{ND} = \{b_n\}_{n \in \mathbb{N}}$ are Dirichlet-Dirichlet and Neumann-Dirichlet 
spectra and $\{\gamma_n\}_{n \in \mathbb{N}}$ are 
point masses of the spectral measure $\mu_{0,0} = \sum_{n \in \mathbb{N}}\gamma_n \delta_{a_n}$. 
\end{theorem}

\begin{proof}
 By representation (\ref{Hrep}) of the $m$-function as a Herglotz integral of the spectral measure, knowing $\gamma_n$ means knowing 
 $Res(m,a_n)$. Therefore, in terms of the $m$-function our claim says that the set of poles, $\{a_n\}_{n \in \mathbb{N}}$, the set of zeros 
 except the index set $A$, $\{b_n\}_{n \in \mathbb{N}\text{\textbackslash} A}$, and the residues with the same index set $A$, 
 $\{Res(m,a_n)\}_{n \in A}$ determine the $m$-function uniquely. Before starting to prove this claim let us briefly list the main steps 
 of the proof. We will use similar ideas to prove our results in non-matching index sets case and for general boundary conditions.\\
 
 \underline{Step 1:} Reduce the claim to the problem of unique recovery of the infinite product 
 \begin{equation*}
  G(z) := -C \prod_{n\in A}\left(\frac{z}{b_{n}}-1\right)\left(\frac{z}{a_n}-1\right)^{-1}
 \end{equation*}
from its sets of poles and residues.

 \underline{Step 2:} Observe that $G(z)$ is a meromorphic Herglotz function and has a representation in terms of its poles, residues and a 
 linear polynomial $dz+e$.
 
 \underline{Step 3:} Show uniqueness of $d$.
 
 \underline{Step 4:} Show uniqueness of $e$.
 
 \underline{Step 5:} Use the representation from Step 2 to get uniqueness of the two spectra and prove the claim by Borg's theorem.\\
 
 \underline{Step 1}\\
 
 From Lemma \ref{lmm1}, the Weyl $m$-function can be represented in terms of $\sigma_{DD}$ and $\sigma_{ND}$,
\begin{equation*}
m(z) = -C \prod_{n\in \mathbb{N}}\left(\frac{z}{b_{n}}-1\right)\left(\frac{z}{a_n}-1\right)^{-1}.
\end{equation*}

Note that for any $k\in A$, we know
\begin{equation}\label{con1}
Res(m,a_k) = C(b_k - a_k)\frac{a_k}{b_k} \prod_{n\in \mathbb{N},n\neq k}\left(\frac{a_k}{b_{n}}-1\right)\left(\frac{a_k}{a_n}-1\right)^{-1}.
\end{equation}

Let $m(z) = F(z)G(z)$, where $F$ and $G$ are two infinite products defined as
\begin{equation*}
 G(z) := -C \prod_{n\in A}\left(\frac{z}{b_{n}}-1\right)\left(\frac{z}{a_n}-1\right)^{-1}, \qquad
 F(z) := \prod_{n\in\mathbb{N} \text{\textbackslash} A}\left(\frac{z}{b_{n}}-1\right)\left(\frac{z}{a_n}-1\right)^{-1}
\end{equation*}

Also note that at any point of $\{a_n\}_{n \in A}$, the infinite product
\begin{equation}\label{con2}
 F(z)=\prod_{n\in\mathbb{N} \text{\textbackslash} A}\left(\frac{z}{b_{n}}-1\right)\left(\frac{z}{a_n}-1\right)^{-1}
\end{equation}
is known.

Conditions (\ref{con1}) and (\ref{con2}) imply that for any $k\in A$, we know

\begin{equation*}
Res(G,a_k) = \frac{Res(m,a_k)}{F(a_k)},
\end{equation*}
i.e. we know all of the poles and residues of $G(z)$, but none of its zeros. We claim that $G(z)$ can be uniquely recovered from this 
data set.\\

\underline{Step 2}\\

Let us observe that $arg(G(z)) = \pi - \sum_{n\in A}\left[arg(z-b_{n}) - arg(z-a_n)\right]$. Since zeros and poles of $G(z)$ 
are real and interlacing, $ 0 ~\textless~ arg(G(z)) ~\textless~ \pi$ for any $z$ in the upper half plane, i.e. $G(z)$ is a 
meromorphic Herglotz function. Therefore by {\u C}ebotarev's theorem, see Theorem \ref{Cebotarev}, $G(z)$ has the representation

\begin{equation}\label{Cebtyperep1}
 G(z) = dz + e + \sum_{n \in A}A_n\left(\frac{1}{a_n-z}-\frac{1}{a_n}\right),
\end{equation}
where $d\geq 0$, $e \in \mathbb{R}$ and $\sum_{n \in A}A_n/a_n^2$ is absolutely convergent.\\

Note that $A_k = -Res(G(z),a_k)$ for any $k \in A$, which means there are only two unknowns on the right hand side of (\ref{Cebtyperep1}), 
namely constants $d$ and $e$.\\

\underline{Step 3}\\

Now let us show uniqueness of $G(z)$ by showing uniqueness of $dz + e$. Let $\widetilde{G}(z)$ be another infinite product sharing same 
properties with $G(z)$, namely:

\begin{itemize}
 \item The infinite product $\widetilde{G}(z)$ is defined as
 \begin{equation*}
  \widetilde{G}(z) := -\widetilde{C} \prod_{n\in A}\left(\frac{z}{\widetilde{b}_{n}}-1\right)\left(\frac{z}{\widetilde{a}_n}-1\right)^{-1},
 \end{equation*}
where $\widetilde{C} ~\textgreater~ 0$, the set of poles $\{\widetilde{a}_n\}_{n \in A}$ satisfies asymptotics (\ref{asy00}) and the set of 
zeros $\{\widetilde{b}_n\}_{n \in A}$ satisfies asymptotics (\ref{asy10}).\\
 \item $G(z)$ and $\widetilde{G}(z)$ share same set of poles with equivalent residues at the corresponding poles, i.e. $\widetilde{a}_k = a_k$ and 
 $Res(\widetilde{G},a_k) =  Res(G,a_k)$ for any $k \in A$.\\
 \item By the equivalence of poles and residues of $G(z)$ and $\widetilde{G}(z)$ and {\u C}ebotarev's theorem, $\widetilde{G}(z)$ has the 
 representation
 \begin{equation}\label{CebTypeRepr1}
  \widetilde{G}(z) = \widetilde{d}z + \widetilde{e} + \sum_{n \in A}A_n\left(\frac{1}{a_n-z}-\frac{1}{a_n}\right),
 \end{equation}
where $\widetilde{d}\geq 0$, $\widetilde{e} \in \mathbb{R}$.
\end{itemize}
Let $k \in A$ and $b_k \neq \widetilde{b}_k$. Since $G(b_k)=0$ and $\widetilde{G}(\widetilde{b}_k)=0$, using representations (\ref{Cebtyperep1}) and 
(\ref{CebTypeRepr1}) we get
\begin{align}
 -db_k-e &= \sum_{n \in A} A_n\left(\frac{1}{a_n-b_k} - \frac{1}{a_n}\right), \label{equ1} \\
 -\tilde{d}\tilde{b}_k-\widetilde{e} &= \sum_{n \in A} A_n\left(\frac{1}{a_n-\tilde{b}_k} - \frac{1}{a_n}\right) \textit{and} \label{equ2} \\
 G(\tilde{b}_k) &= G(\tilde{b}_k) - \widetilde{G}(\tilde{b}_k) = (d-\tilde{d})\tilde{b}_k + e-\tilde{e} \label{equ3}
\end{align}

Replacing $e-\widetilde{e}$ by $G(\widetilde{b}_k) - (d-\widetilde{d})\widetilde{b}_k$ and taking difference of (\ref{equ1}) and (\ref{equ2}) 
we get
\begin{equation*}
 db_k - \tilde{d}\tilde{b}_k - d\tilde{b}_k + \tilde{d}\tilde{b}_k + G(\tilde{b}_k) = 
 \sum_{n \in A} A_n\left(\frac{\widetilde{b}_k - b_k}{(a_n - \widetilde{b}_k)(a_n - b_k)}\right)
\end{equation*}

Dividing both sides by $\widetilde{b}_k(\widetilde{b}_k - b_k)$ we get 
\begin{equation}\label{equ4}
 \frac{-d}{\widetilde{b}_k} + \frac{G(\tilde{b}_k)}{\widetilde{b}_k(\widetilde{b}_k - b_k)} = 
 \sum_{n \in A} \left(\frac{A_n}{\widetilde{b}_k(a_n - \widetilde{b}_k)(a_n - b_k)}\right)
\end{equation}

Note that since $\{a_n\}_{n \in A}$ satisfies asymptotics (\ref{asy00}) and $\{b_n\}_{n \in A}$, $\{\widetilde{b}_n\}_{n \in A}$ satisfy 
asymptotics (\ref{asy10}), the inequality
\begin{equation}\label{eigenvaluesBound}
 |\widetilde{b}_k(a_n - b_k)(a_n - \widetilde{b}_k)|^{-1} \leq |\widetilde{b}_n(a_n - b_n)(a_n - \widetilde{b}_n)|^{-1} \leq 2/a_n^2
\end{equation}
is valid for any $k \in A$, for sufficiently large $n \in A$. In addition, $\sum_{n \in A}A_n/a_n^2$ is 
absolutely convergent. Therefore right hand side of (\ref{equ4}) converges to $0$ 
as $k$ goes to $\infty$. Also note that by (\ref{equ3}), left hand side of (\ref{equ4}) is
\begin{equation}\label{equ5}
 \frac{-d}{\widetilde{b}_k} + \frac{G(\tilde{b}_k)}{\widetilde{b}_k(\widetilde{b}_k - b_k)} =
 \frac{-d}{\widetilde{b}_k} + \frac{G(\tilde{b}_k) - \widetilde{G}(\tilde{b}_k)}{\widetilde{b}_k(\widetilde{b}_k - b_k)} =
 \frac{1}{\tilde{b}_k - b_k}\left[d-\tilde{d}+\frac{e-\tilde{e}}{\tilde{b}_k}\right] - \frac{d}{\tilde{b}_k}.
\end{equation}
Now let us show $\tilde{b}_k - b_k$ converges to $0$ as $k$ goes to $\infty$. Recall that poles of $G$ and $\widetilde{G}$ satisfy asymptotics
\begin{equation*}
  n^2 + \frac{1}{\pi}\int_0^{\pi}q(x)dx + \alpha_n \qquad \textit{and} \qquad
  n^2 + \frac{1}{\pi}\int_0^{\pi}\widetilde{q}(x)dx + \widetilde{\alpha}_n
\end{equation*}
respectively, where $\alpha_n = o(1)$ and $\widetilde{\alpha}_n = o(1)$ as $n \to \infty$. Equivalance of poles of $G$ and $\widetilde{G}$ 
imply equivalence of $\int_0^{\pi}q(x)dx$ and $\int_0^{\pi}\widetilde{q}(x)dx$. Therefore $b_k$ and $\tilde{b}_k$ satisfy asymptotics
\begin{equation*}
 \left(n-\frac{1}{2}\right)^2 + \frac{1}{\pi}\int_0^{\pi}q(x)dx + \beta_n \qquad \textit{and} \qquad
 \left(n-\frac{1}{2}\right)^2 + \frac{1}{\pi}\int_0^{\pi}q(x)dx + \widetilde{\beta}_n,
\end{equation*}
where $\beta_n = o(1)$ and $\widetilde{\beta}_n = o(1)$ as $n \to \infty$. Hence $\tilde{b}_k - b_k = o(1)$ as $k$ goes to $\infty$. 
Therefore by (\ref{equ5}), left hand side of (\ref{equ4}) goes to $\infty$ if $d-\tilde{d} \neq 0$, so we get a contradiction unless 
$d = \tilde{d}$. This implies that $G(z) - \widetilde{G}(z)$ is a real constant, which is $G(0) - \widetilde{G}(0) = \widetilde{C} - C$.\\

\underline{Step 4}\\

Now let us show $\widetilde{C} - C = 0$. Positivity of 
$(\widetilde{b}_k - b_{n})/(\widetilde{b}_k - a_n)$ for all $n\neq k$, which follows from interlacing 
property of $\{a_n\}_{n \in \mathbb{N}}$ and $\{b_n\}_{n \in \mathbb{N}}$, implies 
$sgn(\widetilde{C} - C) = sgn(\widetilde{\beta_k} - \beta_k)$ for all $k\in \mathbb{N}$, i.e. $\{b_n\}_{n \in A}$ and 
$\{\widetilde{b}_n\}_{n \in A}$ are interlacing sequences.\\

Let us assume $\widetilde{C} ~\textgreater~ C$ and wlog the two spectra lie on the positive real line. This implies 
$\widetilde{b}_n ~\textgreater~ b_n$ for all $n \in A$. Observe that $ \prod_{n \in A} \widetilde{b}_n/b_n$ is finite, 
since 
$$
\sum_{n \in A}\frac{\widetilde{b}_n-b_n}{b_n} = \sum_{n \in A}\frac{\widetilde{\beta}_n-\beta_n}{b_n} \leq 
\max_{n \in A}(\widetilde{\beta}_n-\beta_n) \sum_{n \in A}\frac{1}{b_n} ~\textless~ \infty.
$$
Therefore the infinite product $H(z) := G(z)/\widetilde{G}(z)$ is represented as
$$
H(z) := \frac{G(z)}{\widetilde{G}(z)} = \frac{C}{\widetilde{C}}\prod_{n \in A}\frac{z-b_n}{b_n}\frac{\tilde{b}_n}{z-\tilde{b}_n} = 
\frac{C}{\widetilde{C}}\prod_{n \in A}\frac{\tilde{b}_n}{b_n}\prod_{n \in A}\frac{z-b_n}{z-\tilde{b}_n}.
$$
Let us denote the positive real coefficient of $H(z)$ by $ N := (C/\widetilde{C})\prod_{n \in A}\tilde{b}_n/b_n$. 
Then by interlacing property of $\{b_n\}_{n \in A}$ and $\{\widetilde{b}_n\}_{n \in A}$, the infinite product $-H$ is a meromorphic Herglotz 
function, i.e. by Theorem \ref{Cebotarev} it is represented as
\begin{equation}\label{CebtyperepB}
 -H(z) = -N \prod_{n \in A}\frac{z-b_n}{z-\tilde{b}_n} = Dz + E + \sum_{n \in A} B_n \left(\frac{1}{z-\tilde{b}_n} + \frac{1}{\tilde{b}_n}\right),
\end{equation}
where $B_k = -Res(H,\tilde{b}_k)$ and $D,E \in \mathbb{R}$.

Now let us show that $\{B_k/\tilde{b}_k\}_{k \in A}$ is summable.

\begin{align*}
 \left| \frac{B_k}{\widetilde{b}_k} \right| &= 
 N ~\frac{\widetilde{b}_k - b_k}{b_k} \prod_{n \in A, n \neq k} \frac{\widetilde{b}_k - b_n}{\widetilde{b}_k - \widetilde{b}_n} \\
 &\leq N ~\frac{\widetilde{b}_k - b_k}{\widetilde{b}_k} 
 \prod_{n \in A, 1 \leq n \leq k-1} \frac{\widetilde{b}_k - b_n}{\widetilde{b}_k - \widetilde{b}_n} \\
 &= N ~\frac{\widetilde{b}_k - b_k}{\widetilde{b}_k} 
 \prod_{n \in A, 1 \leq n \leq k-1} \left(1 + \frac{\widetilde{b}_n - b_n}{\widetilde{b}_k - \widetilde{b}_n}\right) \\
 &=  N ~\frac{\widetilde{b}_k - b_k}{\widetilde{b}_k} \prod_{n \in A, 1 \leq n \leq k-1} 
 \left(1 + \frac{\widetilde{\beta}_n-\beta_n}{(k-1/2)^2-(n-1/2)^2+\widetilde{\beta}_{k}-\widetilde\beta_n}\right) \\
 &\leq  N ~\frac{\widetilde{b}_k - b_k}{\widetilde{b}_k} \prod_{n \in A, 1 \leq n \leq k-1} 
 \left(1 + \frac{\widetilde{\beta}_n-\beta_n}{(n+1-1/2)^2-(n-1/2)^2+\widetilde{\beta}_{k}-\widetilde\beta_n}\right) \\
 &\leq  N ~\frac{\widetilde{b}_k - b_k}{\widetilde{b}_k} ~M ~\prod_{n = 1}^{k-1}\left(1+\frac{1}{2n}\right),
\end{align*}
for sufficiently large $k$, where $M$ is a real constant independent of $k$. Since $\widetilde{b}_k - b_k = o(1)$, $\widetilde{b}_k = O(k^2)$ 
and $\prod_{n=1}^{k-1}(1+1/2n) = O(\sqrt{k})$ as $k$ goes to $\infty$, $B_k/\widetilde{b}_k = o(1/k^{3/2})$ as $k$ goes to $\infty$ and hence 
$\sum_{n \in A} B_n/\tilde{b}_n$ is absolutely convergent. Then by letting $z$ tend to $-\infty$ in (\ref{CebtyperepB}) we get 
$$
-N = \lim_{t \rightarrow -\infty} \left(Dt + E +  \sum_{n \in A} \frac{B_n}{\widetilde{b}_n} +  \sum_{n \in A} \frac{B_n}{t-\widetilde{b}_n}\right)
$$
and hence $D=0$ and $-N = E + \sum_{n \in A} B_n/\widetilde{b}_n$, i.e. $-H(z)$ has the representation
\begin{equation}
 -H(z) = N - \sum_{n \in A} \frac{B_n}{z-\widetilde{b}_n}.
\end{equation}
Noting that $H(b_k)=0$ and $Res(G,a_k) = Res(\widetilde{G},a_k)$, i.e. $H(a_k) = 1$ for all $k \in A$, we get 
\begin{equation*}
 1 = H(a_k) - H(b_k) = 
 -N + \sum_{n \in A} \frac{B_n}{a_k-\widetilde{b}_n} + N - \sum_{n \in A} \frac{B_n}{b_k-\widetilde{b}_n} =
 \sum_{n \in A} B_n \frac{(b_k - a_k)}{(a_k-\widetilde{b}_n)(b_k-\widetilde{b}_n)}
\end{equation*}
Each term of the infinite sum on the right end is positive, so by letting $k$ go to $\infty$ we get the following contradiction.
\begin{equation*}
 1 = \lim_{k \rightarrow \infty}\sum_{n \in A} B_n \frac{(b_k - a_k)}{(a_k-\widetilde{b}_n)(b_k-\widetilde{b}_n)} 
 = \sum_{n \in A} B_n \lim_{k \rightarrow \infty}\frac{(b_k - a_k)}{(a_k-\widetilde{b}_n)(b_k-\widetilde{b}_n)} = 0
\end{equation*}
Similar arguments give another contradiction, when $\widetilde{C} ~\textless ~C$, so $C = \widetilde{C}$.\\ 

\underline{Step 5}\\

Step 4 implies uniqueness of $dz + e$, i.e. uniqueness of $G(z)$ and hence uniqueness of $\{b_n\}_{n\in A}$. 
After unique recovery of the two spectra $\sigma_{DD} = \{a_n\}_{n \in \mathbb{N}}$ and $\sigma_{ND} = \{b_n\}_{n \in \mathbb{N}}$, the 
potential is uniquely determined by Borg's theorem.
\end{proof} 

\begin{remark}
 If we let $A = \mathbb{N}$, Theorem \ref{DDNDthm} gives Marchenko's theorem with Dirichlet-Dirichlet, Neumann-Dirichlet boundary conditions 
 as a corollary. By letting $A = \emptyset$, we get the statement of Borg's theorem with Dirichlet-Dirichlet, 
 Neumann-Dirichlet boundary conditions.
\end{remark}

\begin{remark}
 Spectral data of Theorem \ref{DDNDthm} can be seen as $\{a_n\}_{n \in \mathbb{N}}$, 
$\{b_n\}_{n \in \mathbb{N}\text{\textbackslash} A}$ and $\{\tau_{\alpha}(a_n)\}_{n \in A}$, where $\{\tau_{\alpha}(a_n)\}_{n \in A}$ is the 
set of norming constants for $\sigma_{DD} = \{a_n\}_{n \in \mathbb{N}}$.
\end{remark}

\subsection{\bf{Non-matching index sets}}
If the known point masses of the spectral measure and unknown eigenvalues of the Neumann-Dirichlet spectrum have different index sets, one 
needs some control over eigenvalues of the Dirichlet-Dirichlet spectrum corresponding to known point masses and unknown part of the 
Neumann-Dirichlet spectrum. In this case we get a {\u C}ebotarev type representation result. Before the statement, let us clarify the 
notations we use. For any subsequence $\displaystyle\{a_{k_n}\}_{n\in \mathbb{N}} \subset \sigma_{DD}$ and 
$\{b_{l_n}\}_{n\in \mathbb{N}} \subset \sigma_{ND}$, by $A_{k_n,m}$ and $A_{k_n}$ we denote the residues at $a_{k_n}$ of partial and 
infinite products, respectively, consisting of these subsequences: 
\begin{align*}
 A_{k_n,m} := Res(G_m,a_{k_n}) &= \frac{a_{k_n}}{b_{l_n}}(a_{k_n}-b_{l_n})\prod_{1\leq j\leq m, j\neq n}\frac{a_{k_j}}{b_{l_j}}\frac{a_{k_n}-b_{l_j}}{a_{k_n}-a_{k_j}},\\
 A_{k_n} := Res(G,a_{k_n}) &= \frac{a_{k_n}}{b_{l_n}}(a_{k_n}-b_{l_n})\prod_{j\in\mathbb{N},j\neq n}\frac{a_{k_j}}{b_{l_j}}\frac{a_{k_n}-b_{l_j}}{a_{k_n}-a_{k_j}},
\end{align*}
where 
\begin{equation*}
G_m(z) := \prod_{n = 1}^{m} \left(\frac{z}{b_{l_n}}-1\right)\left(\frac{z}{a_{k_n}}-1\right)^{-1}, \qquad 
G(z) := \prod_{n\in \mathbb{N}} \left(\frac{z}{b_{l_n}}-1\right)\left(\frac{z}{a_{k_n}}-1\right)^{-1}.
\end{equation*}

Note that these subsequences are ordered according to their indices, i.e. $a_{k_n} ~\textless ~a_{k_{n+1}}$ and 
$b_{l_n} ~\textless ~b_{l_{n+1}}$ for any $n \in \mathbb{N}$. This follows from the asymptotics of the spectra.
\begin{lemma}\label{lmm2} Let $\{a_{k_n}\}_{n\in \mathbb{N}} \subset \sigma_{DD}$ and 
$\{b_{l_n}\}_{n\in \mathbb{N}} \subset \sigma_{ND}$ satisfy following properties:
\begin{itemize}
 \item $\displaystyle \lim_{m \rightarrow \infty} ~\sum_{n=1}^{m} \Big(|A_{k_n,m} - A_{k_n}|/a_{k_n}^2\Big) ~\textless~ \infty$,
 \item $\{A_{k_n}/a_{k_n}^2\}_{n \in \mathbb{N}} \in l^1$.
\end{itemize}

Then 
\begin{equation}\label{Cebtyperep2}
  G(z) = cz^2 + dz + e + \sum_{n\in \mathbb{N}} A_{k_n}\left(\frac{1}{z-a_{k_n}}+\frac{1}{a_{k_n}}\right),
\end{equation}
where $c, d, e$ are real numbers, $A_{k_n}$ is the residue of $G(z)$ at the point $z=a_{k_n}$ and the sum converges 
normally on $\displaystyle \mathbb{C} \text{\textbackslash} \cup_{n \in \mathbb{N}} a_{k_n}$.
\end{lemma}

\begin{proof}
 Let $p(z)$ be the difference of $G(z)$ and the infinite sum on the right hand side of (\ref{Cebtyperep2}). Then, $p(z)$ is an entire function, since the infinite product and 
 the infinite sum share the same set of poles with equivalent degrees and residues. We represent $G_m(z)$ as partial sums:
 \begin{equation*}
  \prod_{n=1}^m \left(\frac{z}{b_{l_n}}-1\right)\left(\frac{z}{a_{k_n}}-1\right)^{-1} = 
  \sum_{n=1}^m A_{k_n,m}\left(\frac{1}{z-a_{k_n}}+\frac{1}{a_{k_n}}\right) + 1,
 \end{equation*}
where $A_{k_n,m} = Res(G_m,a_{k_n})$.

Let ${C_n}$ be the circle with radius $b_{l_n}$ centered at the origin. This sequence of circles satisfy following properties:
 
 \begin{itemize}
 \item $C_n$ omits all the poles $a_{k_n}$.\\
 
 \item Each $C_n$ lies inside $C_{n+1}$.\\
 
 \item The radius of $C_n$, $b_{l_n}$ diverges to infinity as $n$ goes to infinity.\\ 
\end{itemize}
Then,
\begin{align*}
\max_{z\in C_t}\left|\frac{p(z)-1}{b_{l_t}^2}\right| &= \max_{z\in C_t}\left|\frac{G(z)-1 - 
	\sum_{n\in \mathbb{N}} A_{k_n}\left(\frac{1}{z-a_{k_n}}+\frac{1}{a_{k_n}}\right)}{b_{l_t}^2}\right|\\
&= \frac{1}{b_{l_t}^2} \max_{z\in C_t}\lim_{m \to \infty} \left| \sum_{n=1}^m A_{k_n,m}\left(\frac{1}{z-a_{k_n}}+\frac{1}{a_{k_n}}\right) - 
\sum_{n=1}^m A_{k_n}\left(\frac{1}{z-a_{k_n}}+\frac{1}{a_{k_n}}\right)\right|\\
&= \lim_{m \to \infty} \frac{1}{b_{l_t}^2}\max_{z\in C_t}\left| \sum_{n=1}^m (A_{k_n,m}-A_{k_n})\frac{z}{a_{k_n}(z-a_{k_n})}\right|\\
&\leq \lim_{m \to \infty} \frac{1}{b_{l_t}^2}\sum_{n=1}^m |A_{k_n,m}-A_{k_n}|\frac{b_{l_t}}{a_{k_n}|b_{l_t}-a_{k_n}|}\\
&= \lim_{m \to \infty} \sum_{n=1}^m |A_{k_n,m}-A_{k_n}|\frac{1}{a_{k_n}b_{l_t}|b_{l_t}-a_{k_n}|}\\
&\leq \lim_{m \to \infty} \sum_{n=1}^m |A_{k_n,m}-A_{k_n}|\frac{1}{a_{k_n}b_{l_1}|b_{l_1}-a_{k_n}|}\\
&\leq \lim_{m \to \infty} C'\sum_{n=1}^m \frac{|A_{k_n,m}-A_{k_n}|}{a_{k_n}^2} ~\textless~ \infty.
\end{align*}
Note that the second inequality is a consequence of 
$$
\sup_{t \in \mathbb{N}} \Big(b_{l_t}|b_{l_t}-a_{k_n}|\Big)^{-1} \leq \Big(b_{l_1}|b_{l_1}-a_{k_n}|\Big)^{-1},
$$
which follows from asymptotics of $\{a_{n}\}_{n \in \mathbb{N}}$ and $\{b_{n}\}_{n \in \mathbb{N}}$. 
Therefore $|p(z)-1|\leq C''|z|^2$ on the circle $C_t$ for any $t \in \mathbb{N}$, where $C'$ and $C''$ are real numbers. 
By the maximum modulus theorem and the entireness of $p(z)$, we conclude that $p(z)$ is a polynomial of at most second degree. Since 
$G(0), G'(0)$ and $G''(0)$ are real numbers, $c,d,e \in \mathbb{R}$. 

\end{proof}

Using this {\u C}ebotarev type representation we prove our main result in non-matching index sets case with Dirichlet-Dirichlet, 
Neumann-Dirichlet boundary conditions. However, we need extra information of an eigenvalue from $\{b_{l_n}\}_{n\in \mathbb{N}}$.

\begin{theorem}\label{DDNDindex1}
 Let $q \in L^1(0,\pi)$, and $\displaystyle\{a_{k_n}\}_{n\in \mathbb{N}} \subset \sigma_{DD}$, 
$\displaystyle \{b_{l_n}\}_{n\in \mathbb{N}} \subset \sigma_{ND}$ satisfy following properties:
\begin{itemize}
 \item $\displaystyle \lim_{m \rightarrow \infty} ~\sum_{n=1}^{m} \Big(|A_{k_n,m} - A_{k_n}|/a_{k_n}^2\Big) ~\textless~ \infty$,
 \item $\{A_{k_n}/a_{k_n}^2\}_{n \in \mathbb{N}} \in l^1$.
\end{itemize}
Then $\{a_n\}_{n \in \mathbb{N}}$, 
$\{b_n\}_{n \in \mathbb{N}}\text{\textbackslash}\{b_{l_n}\}_{n\in \mathbb{N}\text{\textbackslash}\{s\}}$ and 
$\{\gamma_{k_n}\}_{n \in \mathbb{N}}$ determine the potential $q$ for any $s \in \mathbb{N}$, 
where $\sigma_{DD} = \{a_n\}_{n \in \mathbb{N}}$, $\sigma_{ND} = \{b_n\}_{n \in \mathbb{N}}$ are Dirichlet-Dirichlet and 
Neumann-Dirichlet spectra and $\{\gamma_n\}_{n \in \mathbb{N}}$ are 
point masses of the spectral measure $ \mu_{0,0} = \sum_{n \in \mathbb{N}}\gamma_n \delta_{a_n}$. 
\end{theorem}

 \begin{proof} 
  By representation of the $m$-function as the Herglotz integral of the spectral measure, from $\gamma_n$, we know
  $Res(m,a_n)$. Therefore, in terms of the $m$-function our claim says that the set of poles, $\{a_n\}_{n \in \mathbb{N}}$, the set of zeros 
 except the index set $\{l_n\}_{n \in \mathbb{N}\text{\textbackslash}\{s\}}$, 
 $\{b_{l_s}\} \cup \{b_n\}_{n \in \mathbb{N}\text{\textbackslash} \{l_n\}_{n \in \mathbb{N}}}$, 
 and the residues with the index set $\{k_n\}_{n \in \mathbb{N}}$, $\{Res(m,a_{k_n})\}_{n \in \mathbb{N}}$ determine the 
 $m$-function uniquely.
 
 From Lemma \ref{lmm1}, the Weyl $m$-function can be represented in terms of $\sigma_{DD}$ and $\sigma_{ND}$,
\begin{equation*}
m(z) = -C \prod_{n\in \mathbb{N}}\left(\frac{z}{b_{n}}-1\right)\left(\frac{z}{a_n}-1\right)^{-1}.
\end{equation*}
Note that for any $n\in \mathbb{N}$, we know
\begin{equation}\label{Con1}
Res(m,a_{k_n}) = C(b_{k_n} - a_{k_n})\frac{a_{k_n}}{b_{k_n}} 
\prod_{j\in \mathbb{N},j \neq k_n}\left(\frac{a_{k_n}}{b_j}-1\right)\left(\frac{a_{k_n}}{a_j}-1\right)^{-1}.
\end{equation}
Let $m(z) = F(z)G(z)$, where $F$ and $G$ are two infinite products defined as
\begin{align*}
 G(z) &:= -C \prod_{n\in \mathbb{N}}\left(\frac{z}{b_{l_n}}-1\right)\left(\frac{z}{a_{k_n}}-1\right)^{-1},\\
 F(z) &:= \prod_{n\in\mathbb{N} \text{\textbackslash} \{l_n\}_{n \in \mathbb{N}}} \left(\frac{z}{b_{n}}-1\right)
\prod_{n\in\mathbb{N} \text{\textbackslash} \{k_n\}_{n \in \mathbb{N}}}\left(\frac{z}{a_n}-1\right)^{-1}
\end{align*}
Also note that at any point of $\{a_{k_n}\}_{n \in \mathbb{N}}$, the infinite product
\begin{equation}\label{Con2}
 F(z)=\prod_{n\in\mathbb{N} \text{\textbackslash} \{l_n\}_{n \in \mathbb{N}}} \left(\frac{z}{b_{n}}-1\right)
\prod_{n\in\mathbb{N} \text{\textbackslash} \{k_n\}_{n \in \mathbb{N}}}\left(\frac{z}{a_n}-1\right)^{-1}
\end{equation}
is known. Conditions (\ref{Con1}) and (\ref{Con2}) imply that for any $n\in \mathbb{N}$, we know
\begin{equation*}
Res(G,a_{k_n}) = \frac{Res(m,a_{k_n})}{F(a_{k_n})}.
\end{equation*}

By Lemma \ref{lmm2}, $G(z)$ has the following representation

\begin{equation}\label{Cebtyperep3}
 G(z) = cz^2 + dz + e + \sum_{n\in \mathbb{N}} A_{k_n}\left(\frac{1}{z-a_{k_n}}+\frac{1}{a_{k_n}}\right),
\end{equation}
where $A_{k_n} = Res(G,a_{k_n})$.
 In order to show uniqueness of $G(z)$, let us consider $\widetilde{G}(z)$ similar to the proof of Theorem \ref{DDNDthm}, i.e. 
 $\widetilde{G}(z)$ has the following properties. 
 \begin{itemize}
 \item The infinite product $\widetilde{G}(z)$ is defined as
 \begin{equation*}
  \widetilde{G}(z) := -\widetilde{C} \prod_{n\in \mathbb{N}}\left(\frac{z}{\widetilde{b}_{l_n}}-1\right)\left(\frac{z}{\widetilde{a}_{k_n}}-1\right)^{-1},
 \end{equation*}
where $\widetilde{C} ~\textgreater~ 0$, the set of poles $\{\widetilde{a}_{k_n}\}_{n \in \mathbb{N}}$ satisfies asymptotics (\ref{asy00}) 
and the set of zeros $\{\widetilde{b}_{l_n}\}_{n \in \mathbb{N}}$ satisfies asymptotics (\ref{asy10}). For the given eigenvalues from 
$\sigma_{ND} = \{b_n\}_{n \in \mathbb{N}}$, let $\tilde{b}_n$ be defined as $b_n$, i.e. $\tilde{b}_j := b_j$ for all 
$j \in \mathbb{N}\text{\textbackslash}\{l_n\}_{n \in \mathbb{N}}$. \\
 \item $G(z)$ and $\widetilde{G}(z)$ share same set of poles with equivalent residues at the corresponding poles, i.e. $\widetilde{a}_{k_n} = a_{k_n}$ 
 and $Res(\widetilde{G},a_{k_n}) =  Res(G,a_{k_n})$ for any $n \in \mathbb{N}$.\\
 \item $G(z)$ and $\widetilde{G}(z)$ share one zero, namely $b_{l_s} = \tilde{b}_{l_s}$.\\
 \item By the equivalence of poles and residues of $G(z)$ and $\widetilde{G}(z)$ and Lemma \ref{lmm2}, $\widetilde{G}(z)$ has the 
 representation
 \begin{equation}\label{CebTypeRep3}
  \widetilde{G}(z) = \widetilde{c}z^2 + \widetilde{d}z + \widetilde{e} + \sum_{n \in \mathbb{N}}A_{k_n}\left(\frac{1}{a_{k_n}-z}-\frac{1}{a_{k_n}}\right),
 \end{equation}
where $\widetilde{c}$, $\widetilde{d}$, $\widetilde{e} \in \mathbb{R}$.
\end{itemize}

Let $m \in \mathbb{N}\text{\textbackslash}\{s\}$ and $b_{l_m} \neq \widetilde{b}_{l_m}$. Since $G(b_{l_m})=0$ and $\widetilde{G}(\widetilde{b}_{l_m})=0$, 
using representations (\ref{Cebtyperep3}) and 
(\ref{CebTypeRep3}) we get
\begin{align}
 -cb_{l_m}^2 -db_{l_m}-e &= \sum_{n \in \mathbb{N}} A_{k_n}\left(\frac{1}{a_{k_n}-b_{l_m}} - \frac{1}{a_{k_n}}\right), \label{eq1} \\
 -\tilde{c}\tilde{b}_{l_m}^2 - \tilde{d}\tilde{b}_{l_m}-\widetilde{e} &= \sum_{n \in \mathbb{N}} A_{k_n}\left(\frac{1}{a_{k_n}-\widetilde{b}_{l_m}} - \frac{1}{a_{k_n}}\right) \text{ and} \label{eq2} \\
 G(\tilde{b}_{l_m}) = G(\tilde{b}_{l_m}) - \widetilde{G}(\tilde{b}_{l_m}) &= 
 (c-\tilde{c})\tilde{b}_{l_m}^2 + (d-\tilde{d})\tilde{b}_{l_m} + e-\tilde{e} \label{eq3}
\end{align}
Taking difference of (\ref{eq1}) and (\ref{eq2}) and replacing $e-\widetilde{e}$ by 
$G(\tilde{b}_{l_m}) - (c-\tilde{c})\tilde{b}_{l_m}^2 - (d-\widetilde{d})\widetilde{b}_{l_m}$ we get
\begin{equation*}
 cb_{l_m}^2 - c\tilde{b}_{l_m}^2 + db_{l_m} - d\tilde{b}_{l_m} + G(\tilde{b}_{l_m}) = 
 \sum_{n \in \mathbb{N}} A_{k_n}\left(\frac{\widetilde{b}_{l_m} - b_{l_m}}{(a_{k_n} - \widetilde{b}_{l_m})(a_{k_n} - b_{l_m})}\right)
\end{equation*}

Dividing both sides by $\widetilde{b}_{l_m}(\widetilde{b}_{l_m} - b_{l_m})$ we get 
\begin{equation}\label{eq4}
 \frac{-c(b_{l_m}+\tilde{b}_{l_m})}{\widetilde{b}_{l_m}} + \frac{-d}{\widetilde{b}_{l_m}} + \frac{G(\tilde{b}_{l_m})}{\widetilde{b}_{l_m}(\widetilde{b}_{l_m} - b_{l_m})} = 
 \sum_{n \in \mathbb{N}} \left(\frac{A_{k_n}}{\widetilde{b}_{l_m}(a_{k_n} - \widetilde{b}_{l_m})(a_{k_n} - b_{l_m})}\right)
\end{equation}

Note that since $\{a_n\}_{n \in \mathbb{N}}$ satisfies asymptotics (\ref{asy00}) and $\{b_n\}_{n \in \mathbb{N}}$, $\{\widetilde{b}_n\}_{n \in \mathbb{N}}$ 
satisfy asymptotics (\ref{asy10}), the inequalities
\begin{equation*}
 |\widetilde{b}_{l_m}(a_{k_n} - b_{l_m})(a_{k_n} - \widetilde{b}_{l_m})|^{-1} \leq 
 |\widetilde{b}_{k_n}(a_{k_n} - b_{k_n})(a_{k_n} - \widetilde{b}_{k_n})|^{-1} \leq 2/a_{k_n}^2
\end{equation*}
are valid for any $m \in \mathbb{N}\text{\textbackslash}\{s\}$ and for sufficiently large $n \in \mathbb{N}$. Recall that 
$\widetilde{b}_{k_j} := b_{k_j}$ if ${k_j} \notin \{l_n\}_{n \in \mathbb{N}}$. 
In addition, $\sum_{n \in \mathbb{N}}A_{k_n}/a_{k_n}^2$ is 
absolutely convergent. Therefore right hand side of (\ref{eq4}) converges to $0$ 
as $m$ goes to $\infty$. Also note that by (\ref{eq3}), left hand side of (\ref{eq4}) is
\begin{align*}\label{eq5}
 \frac{-c(b_{l_m}+\tilde{b}_{l_m})-d}{\widetilde{b}_{l_m}} &+ \frac{G(\tilde{b}_{l_m})}{\widetilde{b}_{l_m}(\widetilde{b}_{l_m} - b_{l_m})} =
 \frac{-c(b_{l_m}+\tilde{b}_{l_m})-d}{\widetilde{b}_{l_m}} + 
 \frac{G(\tilde{b}_{l_m}) - \widetilde{G}(\tilde{b}_{l_m})}{\widetilde{b}_{l_m}(\widetilde{b}_{l_m} - b_{l_m})}\\
 &= \frac{-c(b_{l_m}+\tilde{b}_{l_m})-d}{\widetilde{b}_{l_m}} + 
 \frac{1}{\tilde{b}_{l_m} - b_{l_m}}\left[(c-\tilde{c})\tilde{b}_{l_m} + d-\tilde{d}+\frac{e-\tilde{e}}{\tilde{b}_{l_m}}\right].
\end{align*}
Let us observe that 
\begin{equation*}
\lim_{m \rightarrow \infty} \frac{-c(b_{l_m}+\tilde{b}_{l_m})-d}{\widetilde{b}_{l_m}} = -2c.
\end{equation*}
Now let us show $\tilde{b}_{l_m} - b_{l_m}$ converges to $0$ as $m$ goes to $\infty$. Recall that poles of $G$ and $\widetilde{G}$ 
satisfy asymptotics
\begin{equation*}
  k_n^2 + \frac{1}{\pi}\int_0^{\pi}q(x)dx + \alpha_{k_n} \qquad \textit{and} \qquad
  k_n^2 + \frac{1}{\pi}\int_0^{\pi}\widetilde{q}(x)dx + \widetilde{\alpha}_{k_n}
\end{equation*}
respectively, where $\alpha_n = o(1)$ and $\widetilde{\alpha}_n = o(1)$ as $n \to \infty$. Equivalance of poles of $G$ and $\widetilde{G}$ 
imply equivalence of $\int_0^{\pi}q(x)dx$ and $\int_0^{\pi}\widetilde{q}(x)dx$. Therefore $b_{l_m}$ and $\tilde{b}_{l_m}$ satisfy asymptotics
\begin{equation*}
 \left(l_m-\frac{1}{2}\right)^2 + \frac{1}{\pi}\int_0^{\pi}q(x)dx + \beta_{l_m} \qquad \textit{and} \qquad
 \left(l_m-\frac{1}{2}\right)^2 + \frac{1}{\pi}\int_0^{\pi}q(x)dx + \widetilde{\beta}_{l_m},
\end{equation*}
where $\beta_m = o(1)$ and $\widetilde{\beta}_m = o(1)$ as $m \to \infty$. Hence $\tilde{b}_{l_m} - b_{l_m} = o(1)$ as $m$ goes to $\infty$. 
Therefore left hand side of (\ref{eq4}) goes to $\infty$ if $c-\tilde{c} \neq 0$ or $d-\tilde{d} \neq 0$, so we get a 
contradiction unless $c = \tilde{c}$ and $d = \tilde{d}$. This implies that $G(z) - \widetilde{G}(z)$ is a real constant. However, $G(z)$ and $\widetilde{G}(z)$ 
share the zero $b_{l_s}$. This implies uniqueness of $G(z)$ and hence uniqueness of $\{b_{l_n}\}_{n\in \mathbb{N}}$. 
After unique recovery of the two spectra $\sigma_{DD}$ and $\sigma_{ND}$, the potential is uniquely determined by Borg's theorem.  
\end{proof}

We also get the uniqueness result without knowing any point from $\{b_{l_n}\}_{n\in \mathbb{N}}$, but this requires 
absolute convergence of $\prod_{n\in \mathbb{N}}a_{k_{n}}/b_{l_n}$. By absolute convergence of 
$\prod_{n\in \mathbb{N}}a_{k_{n}}/b_{l_n}$ we mean absolute convergence of 
$\sum_{n\in \mathbb{N}}(a_{k_{n}}/b_{l_n} - 1)$. Note that Limit Comparison Test implies that 
$\prod_{n\in \mathbb{N}}a_{k_{n}}/b_{l_n}$ is absolutely convergent if and only if 
$\prod_{n\in \mathbb{N}}b_{l_n}/a_{k_{n}}$ is absolutely convergent. Absolute convergence of 
$\prod_{n\in \mathbb{N}}a_{k_{n}}/b_{l_n}$ also implies the two conditions in Lemma \ref{lmm2}, so in this case 
Lemma \ref{lmm2} can be written in the following form.

\begin{lemma}\label{lmm3} Let $\displaystyle\{a_{k_n}\}_{n\in \mathbb{N}} \subset \sigma_{DD}$ and 
$\displaystyle \{b_{l_n}\}_{n\in \mathbb{N}} \subset \sigma_{ND}$ such that 
$\prod_{n\in \mathbb{N}}(a_{k_{n}}/b_{l_n})$ is absolutely convergent. Then 
\begin{equation*}
  G(z) = cz^2 + dz + e + \sum_{n\in \mathbb{N}} A_{k_n}\left(\frac{1}{z-a_{k_n}}+\frac{1}{a_{k_n}}\right),
\end{equation*}
where $c, d, e$ are real numbers, $A_{k_n}$ is the residue of $G(z)$ at the point $z=a_{k_n}$ and the sum converges 
normally on $\displaystyle \mathbb{C} \text{\textbackslash} \cup_{n \in \mathbb{N}} a_{k_n}$.
\end{lemma}

\begin{proof}
 We will show that absolute convergence of $\prod_{n\in \mathbb{N}}(a_{k_n}/b_{l_n})$ implies the two conditions 
 in Lemma \ref{lmm2}, but first we begin by showing that absolute convergence of $\prod_{n\in \mathbb{N}}(a_{k_n}/b_{l_n})$ implies 
 $\{1/(a_{k_{n}} - b_{l_n})\}_{n \in \mathbb{N}} \in l^1$. Since $\prod_{n\in \mathbb{N}}(a_{k_n}/b_{l_n})$ is absolutely convergent, 
 \begin{equation*}
  \sum_{n\in \mathbb{N}}\left|\frac{a_{k_{n}} - b_{l_n}}{b_{l_n}}\right| = 
  \sum_{n\in \mathbb{N}}\left|\frac{k_n^2 - (l_n - 1/2)^2 + \alpha_{k_n} - \beta_{l_n}}
  {(l_n - 1/2)^2 + (1/\pi)\int_0^{\pi}q(x)dx + \beta_{l_n}}\right| ~\textless~ \infty,
 \end{equation*}
 i.e. $\{(k_n^2 - l_n^2 + l_n)/l_n^2\}_{n \in \mathbb{N}} \in l^1$. Note that $\lim_{n \rightarrow \infty} a_{k_n}/b_{l_n} = 1$ 
 implies $\lim_{n \rightarrow \infty} k_n/l_n = 1$. Therefore 
 \begin{align*}
  \infty &~\textgreater~ \sum_{n\in \mathbb{N}}\left|\frac{k_n^2 - l_n^2 + l_n}{l_n^2}\right|\\ 
  &= \sum_{n\in \mathbb{N}}\frac{k_n + l_n}{l_n}\left|\frac{k_n - l_n + l_n/(k_n + l_n)}{l_n}\right|\\
  &\geq \sum_{n=1}^{N}\left|\frac{k_n - l_n + l_n/(k_n + l_n)}{l_n}\right|  + 
  \sum_{n=N+1}^{\infty}\left|\frac{1/4}{l_n}\right|\\
  &\geq c_1 \sum_{n\in \mathbb{N}}\frac{1}{l_n}
 \end{align*}
where $N \in \mathbb{N}, c_1 ~\textgreater~ 0$, i.e. $\{1/l_n\}_{n \in \mathbb{N}} \in l^1$ and by Limit Comparison Test 
$\{1/k_n\}_{n \in \mathbb{N}} \in l^1$. Therefore $\{1/(a_{k_{n}} - b_{l_n})\}_{n \in \mathbb{N}} \in l^1$, since 
$1/|a_{k_{n}} - b_{l_n}| \leq 1/|a_{k_{n}} - b_{k_n}| = O(1/k_n)$ as $n$ goes to $\infty$.

The partial product $G_N$ defined in the beginning of Section 4.2 can be represented as 
\begin{equation*}
 G_N(z) = \sum_{n = 1}^{N}\frac{A_{k_n,N}}{z - a_{k_n}} + \prod_{n = 1}^{N}\frac{a_{k_n}}{b_{l_n}},
\end{equation*}
and hence
\begin{equation}\label{Gnlimit}
 \lim_{N \rightarrow \infty} \sum_{n = 1}^{N}\frac{A_{k_n,N}}{a_{k_n}} = 
 \lim_{N \rightarrow \infty} \Big[\prod_{n = 1}^{N}\frac{a_{k_n}}{b_{l_n}} - G_N(0)\Big] = 
 \prod_{n \in \mathbb{N}}\frac{a_{k_n}}{b_{l_n}} - 1 \in \mathbb{R}.
\end{equation}
Since $\{1/a_{k_n}\}_{n \in \mathbb{N}} \in l^1$, existence of this limit implies 
$\lim_{N \rightarrow \infty}\sum_{n = 1}^{N}|A_{k_n,N}/a_{k_n}^2|$ exists.\\

Now we are ready to prove the first assumption in Lemma \ref{lmm2}, i.e.
$$
\displaystyle \lim_{N \rightarrow \infty} ~\sum_{n=1}^{N} \Big(|A_{k_{n,N}} - A_{k_n}|/a_{k_n}^2\Big) ~\textless~ \infty.
$$
For $n ~\textless~ N$, let us define 
$$
P_{k_n,N} := \prod_{m = N+1}^{\infty}\frac{a_{k_m}}{b_{l_m}}\frac{a_{k_n}-b_{l_m}}{a_{k_n}-a_{k_m}}.
$$
Then 
\begin{equation}\label{AknDifference}
 \frac{|A_{k_{n,N}} - A_{k_n}|}{a_{k_n}^2} = 
 \Bigg|\left(\frac{A_{k_{n,N}}}{a_{k_n}}\right)\left(\frac{a_{k_n}-b_{l_n}}{a_{k_n}}[1-P_{k_n,N}]\right)\left(\frac{1}{a_{k_n}-b_{l_n}}\right)\Bigg|
\end{equation}
Using (\ref{Gnlimit}), absoulte convergence of $\prod_{n \in \mathbb{N}}(a_{k_n}/b_{l_n})$ and hence absolute convergence of 
$\sum_{n \in \mathbb{N}}[(a_{k_n}-b_{l_n})/a_{k_n}]$ we get that the limits
\begin{equation*}
 \lim_{N \rightarrow \infty}\sum_{n=1}^{N}\frac{A_{k_{n,N}}}{a_{k_n}} \quad \text{and} \quad 
 \lim_{N \rightarrow \infty}\sum_{n=1}^{N}\left(\frac{a_{k_n}-b_{l_n}}{a_{k_n}}[1-P_{k_n,N}]\right) \quad \text{converge.}
\end{equation*}
Recall that we have also showed $\{1/(a_{k_{n}} - b_{l_n})\}_{n \in \mathbb{N}} \in l^1$. Therefore by (\ref{AknDifference}) we get the 
first assumption in Lemma \ref{lmm2},
\begin{equation*}
 \lim_{N \rightarrow \infty}\sum_{n = 1}^N \frac{|A_{k_{n,N}} - A_{k_n}|}{a_{k_n}^2} ~\textless~ \infty.
\end{equation*}
After recalling that we showed existence of $\lim_{N \rightarrow \infty}\sum_{n = 1}^{N}|A_{k_n,N}/a_{k_n}^2|$, we get the second assumption 
in Lemma \ref{lmm2}, i.e. $\{A_{k_n}/a_{k_n}^2\}_{n \in \mathbb{N}} \in l^1$ as follows:
\begin{equation*}
 \lim_{N \rightarrow \infty}\sum_{n=1}^N\frac{|A_{k_n}|}{a_{k_n}^2} \leq 
 \lim_{N \rightarrow \infty}\sum_{n=1}^N\frac{|A_{k_n}-A_{k_n,N}|}{a_{k_n}^2} + 
 \lim_{N \rightarrow \infty}\sum_{n=1}^N\frac{|A_{k_n,N}|}{a_{k_n}^2} ~\textless~ \infty.
\end{equation*}
Now using Lemma \ref{lmm2} we get the desired result.

\end{proof}

\begin{theorem}\label{DDNDindex2}
Let $q \in L^1(0,\pi)$ and $\displaystyle\{a_{k_n}\}_{n\in \mathbb{N}} \subset \sigma_{0,0}$, 
$\displaystyle \{b_{l_n}\}_{n\in \mathbb{N}} \subset \sigma_{\pi/2,0}$ such that $\prod_{n\in \mathbb{N}}(a_{k_{n}}/b_{l_n})$ 
is absolutely convergent. Then $\{a_n\}_{n \in \mathbb{N}}$, $\{b_n\}_{n \in \mathbb{N}}\text{\textbackslash}\{b_{l_n}\}_{n\in \mathbb{N}}$ 
and $\{\gamma_{k_n}\}_{n \in \mathbb{N}}$ determine the potential $q$, where 
$\sigma_{DD} = \{a_n\}_{n \in \mathbb{N}}$, $\sigma_{ND} = \{b_n\}_{n \in \mathbb{N}}$ are Dirichlet-Dirichlet and Neumann-Dirichlet spectra 
and $\{\gamma_n\}_{n \in \mathbb{N}}$ are point masses of the spectral measure 
$\mu_{0,0} = \sum_{n \in \mathbb{N}}\gamma_n \delta_{a_n}$. 
\end{theorem}

\begin{proof}
 One can use Lemma \ref{lmm3} and follow the proof of Theorem \ref{DDNDindex1} until the last step, i.e. showing uniqueness of the two 
 spectra after obtaining $G(z) - \widetilde{G}(z)$ is a real constant, so let us show $G(z) - \widetilde{G}(z) = 0$. The main differences 
 in this case are that $G$ and $\widetilde{G}$ do not share any zero and the infinite products 
 $\prod_{n\in \mathbb{N}}(a_{k_{n}}/b_{l_n})$ and $\prod_{n\in \mathbb{N}}(a_{k_{n}}/\tilde{b}_{l_n})$ are absolutely convergent. Let us 
 recall that the infinite products $G$ and $\widetilde{G}$ have the following representations:
 \begin{align*}
  G(z) = cz^2 + dz - C + \sum_{n \in \mathbb{N}}A_{k_n}\left(\frac{1}{z-a_{k_n}}+\frac{1}{a_{k_n}}\right),\\
  \widetilde{G}(z) = cz^2 + dz -\widetilde{C} + \sum_{n \in \mathbb{N}}A_{k_n}\left(\frac{1}{z-a_{k_n}}+\frac{1}{a_{k_n}}\right).
 \end{align*}
 Therefore by taking difference of $G(z)$ and $\widetilde{G}(z)$ we get $G(z) + C = \widetilde{G}(z) + \widetilde{C}$, i.e.
 \begin{equation}\label{infprodrep}
  -C \prod_{n\in \mathbb{N}}\left(\frac{z}{b_{l_n}}-1\right)\left(\frac{z}{a_{k_n}}-1\right)^{-1} + C = 
  -\widetilde{C} \prod_{n\in \mathbb{N}}\left(\frac{z}{\widetilde{b}_{l_n}}-1\right)\left(\frac{z}{a_{k_n}}-1\right)^{-1} + \widetilde{C}.
 \end{equation} 
Note that since the infinite products $\prod_{n\in \mathbb{N}}(a_{k_{n}}/b_{l_n})$, $\prod_{n\in \mathbb{N}}(a_{k_{n}}/\tilde{b}_{l_n})$ 
are absolutely convergent and the two spectra $\{a_n\}_{n \in \mathbb{N}}$, $\{b_n\}_{n \in \mathbb{N}}$ lie on the positive real line, 
the infinite products on the two sides of (\ref{infprodrep}) are uniformly convergent on the second quadrant. Hence by letting $z$ go to 
infinity on the second quadrant we get
\begin{equation}\label{Cequation}
 -C \prod_{n\in \mathbb{N}}\frac{a_{k_n}}{b_{l_n}} + C = 
  -\widetilde{C} \prod_{n\in \mathbb{N}}\frac{a_{k_n}}{\widetilde{b}_{l_n}} + \widetilde{C}.
\end{equation}

Recall that $\prod_{n \in \mathbb{N}} \widetilde{b}_{l_n}/b_{l_n}$ is finite, 
since 
$$
\sum_{n \in \mathbb{N}}\frac{|\widetilde{b}_{l_n}-b_{l_n}|}{b_{l_n}} = 
\sum_{n \in \mathbb{N}}\frac{|\widetilde{\beta}_{l_n}-\beta_{l_n}|}{b_{l_n}} \leq 
\max_{n \in \mathbb{N}}|\widetilde{\beta}_{l_n}-\beta_{l_n}| \sum_{n \in \mathbb{N}}\frac{1}{b_{l_n}} ~\textless~ \infty.
$$
Therefore the infinite product $H(z) := G(z)/\widetilde{G}(z)$ is represented as
$$
H(z) := \frac{G(z)}{\widetilde{G}(z)} = \frac{C}{\widetilde{C}}\prod_{n \in \mathbb{N}}\frac{z-b_{l_n}}{b_{l_n}}\frac{\tilde{b}_{l_n}}{z-\tilde{b}_{l_n}} = 
\frac{C}{\widetilde{C}}\prod_{n \in \mathbb{N}}\frac{\tilde{b}_{l_n}}{b_{l_n}}\prod_{n \in \mathbb{N}}\frac{z-b_{l_n}}{z-\tilde{b}_{l_n}}.
$$

We know that $G$ and $\widetilde{G}$ share same poles with equivalent residues at the corresponding poles. Therefore for any $m \in \mathbb{N}$
\begin{equation}\label{ratioRes}
 1 = H(a_{k_m}) = \frac{C}{\widetilde{C}}\prod_{n \in \mathbb{N}}\frac{\tilde{b}_{l_n}}{b_{l_n}}
 \prod_{n \in \mathbb{N}}\frac{a_{k_m}-b_{l_n}}{a_{k_m}-\tilde{b}_{l_n}}.
\end{equation}

Now let us find the limit of the infinite product on the right end of (\ref{ratioRes}) as $m$ goes to $\infty$. 
This infinite product is uniformly convergent if and only if the infinite sum 
\begin{equation}
 \sum_{n \in \mathbb{N}} \left(\frac{a_{k_m}-b_{l_n}}{a_{k_m}-\tilde{b}_{l_n}} - 1\right) = 
 \sum_{n \in \mathbb{N}} \frac{\tilde{b}_{l_n} - b_{l_n}}{a_{k_m} - \tilde{b}_{l_n}}
\end{equation}
is uniformly convergent.
Note that asymptotics of the two spectra imply $\tilde{b}_{l_j} - b_{l_j} = o(1)$ as $j$ goes to infinity. Then the asymptotics of 
$\{a_{k_n}\}_{n \in \mathbb{N}}, \{b_{l_n}\}_{n \in \mathbb{N}}$ and $\{\tilde{b}_{l_n}\}_{n \in \mathbb{N}}$ together with absolute 
convergence of the infinite products $\prod_{n\in \mathbb{N}}(a_{k_{n}}/b_{l_n})$, $\prod_{n\in \mathbb{N}}(a_{k_{n}}/\tilde{b}_{l_n})$ imply 
that 
\begin{equation}
 \sum_{n \in \mathbb{N}} \left|\frac{\tilde{b}_{l_n} - b_{l_n}}{a_{k_m} - \tilde{b}_{l_n}}\right| \leq 
 \sum_{n \in \mathbb{N}} \left|\frac{\tilde{b}_{l_n} - b_{l_n}}{a_{k_n} - \tilde{b}_{l_n}}\right| ~\textless~ \infty,
\end{equation}
since $\{1/(a_{k_n} - \tilde{b}_{l_n})\}_{n \in \mathbb{N}} \in l^1$ as we discussed in the proof of Lemma \ref{lmm3}. 
Therefore by letting $m$ go to $\infty$ in (\ref{ratioRes}) we get 
$\widetilde{C}/C = \prod_{j \in \mathbb{N}}\tilde{b}_{l_j}/b_{l_j}$. If we define 
$\gamma := \prod_{n\in \mathbb{N}}a_{k_n}/b_{l_n}$ and 
$\widetilde{\gamma} := \prod_{n\in \mathbb{N}}a_{k_n}/\widetilde{b}_{l_n}$, we get 
$\widetilde{C}/C = \gamma/\widetilde{\gamma}$. This identity and (\ref{Cequation}) imply 
$\displaystyle \frac{\gamma - 1}{\widetilde{\gamma} - 1} = \frac{\gamma}{\widetilde{\gamma}}$ and hence $\gamma = \widetilde{\gamma}$. 
Therefore $C = \widetilde{C}$. This implies uniqueness of $G(z)$ and hence uniqueness of $\{b_{l_n}\}_{n\in \mathbb{N}}$. 
After unique recovery of the two spectra $\sigma_{DD}$ and $\sigma_{ND}$, the potential is uniquely determined by Borg's theorem.
\end{proof}

\subsection{\bf{General boundary conditions}}

As discussed in Section 2.2, the Weyl $m$-function for the Schr\"{o}dinger equation 
\begin{equation}\label{Scheq2}
Lu = -u''+qu = zu
\end{equation}
with boundary conditions 
\begin{align}
 &u(0)\cos\alpha - u'(0)\sin\alpha = 0 \label{BCalpha}\\
 &u(\pi)\cos\beta + u'(\pi)\sin\beta = 0, \label{BCbeta}
\end{align}
is defined as $\displaystyle m_{\alpha,\beta}(z) = \frac{\cos(\alpha)u_z'(0) + \sin(\alpha)u_z(0)}{-\sin(\alpha)u_z'(0) + \cos(\alpha)u_z(0)}$, 
where $u_z(t)$ is a solution of (\ref{Scheq2}) satisfying (\ref{BCbeta}) and $\alpha,\beta \in [0,\pi)$. 
In order to prove our result with boundary conditions (\ref{BCalpha}) and (\ref{BCbeta}) we need to consider more general $m$-functions. 
Recall that we have defined the $m$-function in Section 2.2 by introducing two solutions $s_z(t)$ and $c_z(t)$ of (\ref{Scheq2}) satisfying the initial 
conditions 
\begin{align*}
&s_z(0) = \sin(\alpha), \quad s'_z(0) = \cos(\alpha)\\
&c_z(0) = \cos(\alpha), \quad c'_z(0) = -\sin(\alpha)
\end{align*}
and $u_z(t)$, a solution of ($\ref{Scheq2}$) with boundary conditions $u_z(\pi) = \sin\beta$, $u_z'(\pi) = -\cos\beta$. 
The same steps to define the $m$-function as in Section 2.2 can be followed if $c_{z}(t)$ is a linearly independent solution with 
$W(c_z,s_z) = 1$. Therefore we introduce two solutions $s_z(t)$ and $c_z(t)$ of (\ref{Scheq2}) satisfying the initial conditions 
\begin{align*}
s_z(0) = \sin(\alpha_2), \quad &s'_z(0) = \cos(\alpha_2)\\
c_z(0) = \frac{\sin(\alpha_1)}{\sin(\alpha_1-\alpha_2)}, \quad &c'_z(0) = \frac{\cos(\alpha_1)}{\sin(\alpha_1-\alpha_2)}
\end{align*}
for $\alpha_1,\alpha_2 \in [0,\pi)$, $\sin(\alpha_1-\alpha_2) \neq 0$ and same $u_z(t)$. 
Then we can define the $m$-function $m_{\alpha_1,\alpha_2,\beta}$.
\begin{definition}
 The $m$-function $m_{\alpha_1,\alpha_2,\beta}$ is defined as
 \begin{equation*}
    m_{\alpha_1, \alpha_2, \beta}(z) 
    := \frac{1}{\sin(\alpha_2-\alpha_1)} \left[\frac{-\sin(\alpha_1)u_z'(0) \
    + \cos(\alpha_1)u_z(0)}{-\sin(\alpha_2)u_z'(0) + \cos(\alpha_2)u_z(0)}\right],
\end{equation*}
where $\alpha_1,\alpha_2,\beta \in [0,\pi)$, $\sin(\alpha_2 - \alpha_1) \neq 0$ and $u_z(t)$ is a solution of ($\ref{Scheq2}$) with boundary 
conditions $u_z(\pi) = \sin\beta$, $u_z'(\pi) = -\cos\beta$. 
\end{definition}

\begin{remark}
The $m$-function $m_{\alpha,\beta}$ we discussed in Section 2.2 is obtained by letting $\alpha_1 = \alpha - \pi/2$ and 
$\alpha_2 = \alpha$, i.e. $m_{\alpha - \frac{\pi}{2},\alpha,\beta}(z) = m_{\alpha,\beta}(z)$.
\end{remark}

The $m$-function $m_{\alpha_1,\alpha_2,\beta}(z)$ is a meromorphic Herglotz function having real zeros 
on $\sigma_{\alpha_1,\beta}$ and real poles on $\sigma_{\alpha_2,\beta}$, which are interlacing. 
It is a meromorphic Herglotz function, since $m_{0,\beta}(z) = u_z'(0)/u_z(0)$ is a meromorphic Herglotz function and 
$sgn[\Im(m_{\alpha_1,\alpha_2,\beta}(z))] = sgn[\Im(m_{0,\beta}(z))]$. Therefore Herglotz representation theorem implies
\begin{equation*}
 m_{\alpha_1,\alpha_2,\beta}(z) = az + b + \int\left[\frac{1}{t-z}-\frac{t}{1+t^2}\right]d\mu_{\alpha_1,\alpha_2,\beta}(t),
\end{equation*}
where $a,b\in \mathbb{R}$ and $\mu_{\alpha_1,\alpha_2,\beta}$ is a positive discrete Poisson-summable measure supported on the spectrum 
$\sigma_{\alpha_2,\beta}$. Let us call $\mu_{\alpha_1,\alpha_2,\beta}$ the spectral measure corresponding to $(\alpha_1,\alpha_2,\beta)$. 
Now we prove our results with general boundary conditions.

\begin{theorem}\label{GBCthm} 
Let $q \in L^1(0,\pi)$, $A \subset \mathbb{N}$, $\sin(\alpha_2 - \alpha_1) \neq 0$ and 
$\alpha_1,\alpha_2,\beta \in [0,\pi)$. Then $\{a_n\}_{n \in \mathbb{N}}$, 
$\{b_n\}_{n \in \mathbb{N}\text{\textbackslash} A}$ and $\{\gamma_n\}_{n \in A}$ determine the potential $q$, 
where $\sigma_{\alpha_2,\beta} = \{a_n\}_{n \in \mathbb{N}}$, $\sigma_{\alpha_1,\beta} = \{b_n\}_{n \in \mathbb{N}}$ are two spectra
and $\{\gamma_n\}_{n \in \mathbb{N}}$ are 
point masses of the corresponding spectral measure $\mu_{\alpha_1,\alpha_2,\beta} = \sum_{n \in \mathbb{N}}\gamma_n \delta_{a_n}$. 
\end{theorem}

\begin{proof}
Wlog let $a_n$ and $b_n$ be positive for all $n$. We follow the arguments we used in the proofs of Lemma \ref{lmm1} and 
Theorem \ref{DDNDthm}, but there are two differences: asymptotics of the two spectra, depending on $\alpha_1,\alpha_2,\beta$ and 
hence the order relation betweeen $a_n$ and $b_n$. Thus, we consider the following cases.\\

  (i) \underline{$\alpha_1 \neq 0$, $\alpha_2 \neq 0$, $\alpha_1 ~\textgreater~ \alpha_2$ :}\\
  
  When $\beta \neq 0$, the two spectra $\sigma_{\alpha_2,\beta} = \{a_n\}_{n \in \mathbb{N}}$ and 
  $\sigma_{\alpha_1,\beta} = \{b_n\}_{n \in \mathbb{N}}$ satisfy the asymptotics (\ref{asy11}) and hence $a_n ~\textgreater~ b_n$ for 
  all $n\in \mathbb{N}$. Therefore using the proof of Lemma \ref{lmm1}, $m_{\alpha_1,\alpha_2,\beta}(z)$ can be represented as 
  (\ref{mfcn2}). Using this representation and {\u C}ebotarev's theorem as we discussed in the proof of Theorem \ref{DDNDthm}, 
  the meromorphic Herglotz function $G(z)$ defined as 
  \begin{equation}\label{defG}
   G(z):=-C \prod_{n\in A}\left(\frac{z}{b_{n}}-1\right)\left(\frac{z}{a_n}-1\right)^{-1}
  \end{equation}
has the following representation:
  \begin{equation*}
 G(z) = dz + e + \sum_{n \in A}A_n\left(\frac{1}{z-a_n}+\frac{1}{a_n}\right).
\end{equation*}
 Only unknown constants on the right hand side are $d$ and $e$. In order to show uniqueness of the linear term $dz +e $, let us introduce 
 $\widetilde{G}(z)$ as we did in the proof of Theorem~\ref{DDNDthm}:
\begin{itemize}
 \item The infinite product $\widetilde{G}$ is defined as
 \begin{equation*}
  \widetilde{G}(z) := -\widetilde{C} \prod_{n\in A}\left(\frac{z}{\widetilde{b}_{n}}-1\right)\left(\frac{z}{\widetilde{a}_n}-1\right)^{-1},
 \end{equation*}
where $\widetilde{C} ~\textgreater~ 0$, the set of poles $\{\widetilde{a}_n\}_{n \in A}$ and the set of 
zeros $\{\widetilde{b}_n\}_{n \in A}$ satisfy asymptotics (\ref{asy11}).\\
 \item $G$ and $\widetilde{G}$ share same set of poles with equivalent residues at the corresponding poles, i.e. $\widetilde{a}_k = a_k$ and 
 $Res(\widetilde{G},a_k) =  Res(G,a_k)$ for any $k \in A$.\\
 \item By the equivalence of poles and residues of $G$ and $\widetilde{G}$ and {\u C}ebotarev's theorem, $\widetilde{G}(z)$ has the 
 representation
 \begin{equation*}\label{CebTypeRep1}
  \widetilde{G}(z) = \widetilde{d}z + \widetilde{e} + \sum_{n \in A}A_n\left(\frac{1}{z-a_n}+\frac{1}{a_n}\right),
 \end{equation*}
where $\widetilde{d}\geq 0$, $\widetilde{e} \in \mathbb{R}$.
\end{itemize}

Therefore difference of $G$ and $\widetilde{G}$ is a linear polynomial, i.e.
\begin{equation}\label{GminGt}
 G(z) - \widetilde{G}(z) = (d-\widetilde{d})z + e - \widetilde{e}
\end{equation}

Note that since $\{a_n\}_{n \in \mathbb{N}}$, $\{b_n\}_{n \in \mathbb{N}}$ and $\{\tilde{b}_n\}_{n \in \mathbb{N}}$ are subsets of 
$(0,\infty)$ and satisfy asymptotics (\ref{asy11}), for any $x \in (-\infty , 0)$ we get
\begin{align*}
 |G(x)-\widetilde{G}(x)| &\leq \left|C \prod_{n\in A}\left(\frac{x}{b_{n}}-1\right)\left(\frac{x}{a_n}-1\right)^{-1}\right| + 
 \left| \widetilde{C} \prod_{n\in A}\left(\frac{x}{\widetilde{b}_{n}}-1\right)\left(\frac{x}{a_n}-1\right)^{-1}\right|\\
 &\leq C \prod_{n\in A}\frac{a_n}{b_{n}} + \widetilde{C} \prod_{n\in A}\frac{a_n}{\widetilde{b}_{n}} ~\textless~ \infty.
\end{align*}
Convergence of the infinite product $\prod_{n\in A}a_n/b_{n}$ follows from the fact that
\begin{equation*}
 \sum_{n\in A} \frac{a_n-b_n}{b_n} \leq M \sum_{n\in A}\frac{1}{n^2},
\end{equation*}
for some $M ~\textless~ \infty$, since asymptotics (\ref{asy11}) imply $|a_n-b_n| \leq M_1$ for some $M_1 ~\textless~ \infty$ independent 
of $n$ and $a_n = n^2 + o(n^2)$, $b_n = n^2 + o(n^2)$ as $n$ goes to infinity. Therefore 
\begin{equation*}
 \lim_{x \rightarrow -\infty} |(d-\widetilde{d})x + e - \widetilde{e}| = \lim_{x \rightarrow -\infty} |G(x)-\widetilde{G}(x)| = 
 \left|C \prod_{n\in A}\frac{a_n}{b_{n}} - \widetilde{C} \prod_{n\in A}\frac{a_n}{\widetilde{b}_{n}}\right| ~\textless~ \infty, 
\end{equation*}
so we get a contradiction unless $d = \tilde{d}$. This implies that $G(z) - \widetilde{G}(z)$ is a real constant, which is 
$G(0) - \widetilde{G}(0) = \widetilde{C} - C$. In order to show $\widetilde{C} = C$, we follow exactly the same arguments used in the proof 
of Theorem \ref{DDNDthm}.

This gives uniqueness of $G(z)$ and hence uniqueness of $\{b_n\}_{n\in A}$. 
After unique recovery of the two spectra $\sigma_{\alpha_2,\beta}$ and $\sigma_{\alpha_1,\beta}$, Levinson's theorem uniquely determines 
the potential.\\

When $\beta = 0$, one can apply same arguments. The only difference appears in asymptotics of 
$\sigma_{\alpha_2,\beta} = \{a_n\}_{n \in \mathbb{N}}$ and $\sigma_{\alpha_1,\beta} = \{b_n\}_{n \in \mathbb{N}}$, which does not affect the 
result.\\
 
 (ii) \underline{$\alpha_1 \neq 0$, $\alpha_2 = 0$, $\beta = 0$ :}\\ 
  
  The two spectra $\sigma_{\alpha_2,\beta} = \{a_n\}_{n \in \mathbb{N}}$ and 
  $\sigma_{\alpha_1,\beta} = \{b_n\}_{n \in \mathbb{N}}$ satisfy the asymptotics (\ref{asy00}) and (\ref{asy10}) respectively. One then 
  obtains the result by following the proofs of Lemma~\ref{lmm1} and Theorem~\ref{DDNDthm}.\\
  
  (iii) \underline{$\alpha_1 \neq 0$, $\alpha_2 = 0$, $\beta \neq 0$ :}\\ 
  
   The two spectra $\sigma_{\alpha_2,\beta} = \{a_n\}_{n \in \mathbb{N}}$ and 
  $\sigma_{\alpha_1,\beta} = \{b_n\}_{n \in \mathbb{N}}$ satisfy the asymptotics (\ref{asy01}) and (\ref{asy11}) respectively, which is similar to the 
  previous case.\\
  
   (iv) \underline{$\alpha_1 \neq 0$, $\alpha_2 \neq 0$, $\alpha_1 ~\textless~ \alpha_2$ or 
    $\alpha_1 = 0$, $\alpha_2 \neq 0$, $\beta \neq 0$ or $\alpha_1 = 0$, $\alpha_2 \neq 0$, $\beta = 0$ :}\\
  
  In all of these three cases, $a_n ~\textless ~b_n$ for all $n\in \mathbb{N}$. 
  Therefore using the proof of Lemma \ref{lmm1}, $m_{\alpha_1,\alpha_2,\beta}(z)$ can be represented as 
  \begin{equation*}
   m_{\alpha_1,\alpha_2,\beta}(z) = C \prod_{n\in \mathbb{N}}\left(\frac{z}{b_{n}}-1\right)\left(\frac{z}{a_n}-1\right)^{-1}.
  \end{equation*}
  In order to represent $G(z)$ as (\ref{defG}), an extra factor is required, so we shift 
  indices of $b_n$ up by one inside $A$ and let $b_1$ be a positive real number less than $a_1$, assuming wlog $1 \in A$. 
  Then $\frac{z-b_1}{b_1}G(z)$ can be represented as (\ref{defG}). 
  Using this representation and {\u C}ebotarev's theorem, the meromorphic Herglotz function 
  $\frac{z-b_1}{b_1}G(z)$ has the following representation:
  
  \begin{equation*}
 \left(\frac{z-b_1}{b_1}\right)G(z) = az + b + \sum_{n \in A}A_n\left(\frac{1}{a_n-z}-\frac{1}{a_n}\right).
\end{equation*}
  Therefore if we introduce $\widetilde{G}(z)$ similar to the previous cases, then $\frac{z-b_1}{b_1}G(z)$ and 
  $ \frac{z-b_1}{b_1}\widetilde{G}(z)$ share the same set of poles $\{a_n\}_{n \in A}$ with the same residues $\{-A_n\}_{n \in A}$ and 
  have the set of zeros $\{b_n\}_{n \in A}$ and 
$\{b_1\} \cup \{\widetilde{b}_n\}_{n \in A\text{\textbackslash}\{1\}}$ 
respectively, so the difference of $\frac{z-b_1}{b_1}G(z)$ and $ \frac{z-b_1}{b_1}\widetilde{G}(z)$ is a linear polynomial with real 
coefficients and hence $G(z) - \widetilde{G}(z)$ is a real constant, which is 
$G(0) - \widetilde{G}(0) = \widetilde{C} - C$. In order to show $\widetilde{C} = C$, we follow exactly the same arguments used in the proof 
of Theorem \ref{DDNDthm}.

This implies uniqueness of $G(z)$ and hence uniqueness of $\{b_n\}_{n\in A}$. 
After unique recovery of the two spectra $\sigma_{\alpha_2,\beta}$ and $\sigma_{\alpha_1,\beta}$, Levinson's theorem uniquely determines 
the potential.
\end{proof}

\begin{remark}
 Theorem~\ref{GBCthm} gives Marchenko's theorem with the $m$-function $m_{\alpha_1,\alpha_2,\beta}$ as a corollary if we let $A = \mathbb{N}$. 
 By letting $A = \emptyset$, we get the statement of Levinson's theorem.
\end{remark}

For the non-matching index sets case, let us recall the definitions of $A_{k_n,m}$ and $A_{k_n}$:
\begin{align*}
 A_{k_n,m} &:= \frac{a_{k_n}}{b_{l_n}}(a_{k_n}-b_{k_n})\prod_{j=1,j\neq n}^m\frac{a_{k_j}}{b_{l_j}}\frac{a_{k_n}-b_{l_j}}{a_{k_n}-a_{k_j}},\\
 A_{k_n} &:= \frac{a_{k_n}}{b_{l_n}}(a_{k_n}-b_{k_n})\prod_{j=1,j\neq n}^{\infty}\frac{a_{k_j}}{b_{l_j}}\frac{a_{k_n}-b_{l_j}}{a_{k_n}-a_{k_j}}.
\end{align*}

We can prove Theorem \ref{DDNDindex1} and Theorem \ref{DDNDindex2} with general boundary conditions following the same proofs. However, if 
boundary conditions $\alpha_1$ and $\alpha_2$ are nonzero, then we need that eventually the two index sets $\{k_n\}_{n \in \mathbb{N}}$ and 
$\{l_n\}_{n \in \mathbb{N}}$ have 
no common element.

\begin{theorem}\label{GBCindex1} 
 Let $q \in L^1(0,\pi)$, $\sin(\alpha_2 - \alpha_1) \neq 0$, $\alpha_1,\alpha_2,\beta \in [0,\pi)$ and 
 $\displaystyle\{a_{k_n}\}_{n\in \mathbb{N}} \subset \sigma_{\alpha_2,\beta}$, 
$\displaystyle \{b_{l_n}\}_{n\in \mathbb{N}} \subset \sigma_{\alpha_1,\beta}$ satisfy following properties:
\begin{itemize}
 \item $\displaystyle \lim_{m \rightarrow \infty} ~\sum_{n=1}^{m} \Big(|A_{k_n,m} - A_{k_n}|/a_{k_n}^2\Big) ~\textless~ \infty$,
 \item $\{A_{k_n}/a_{k_n}^2\}_{n \in \mathbb{N}} \in l^1$.\\
\end{itemize}

(i) If $\alpha_1 = 0$ or $\alpha_2 = 0$, then $\{a_n\}_{n \in \mathbb{N}}$, 
$\{b_n\}_{n \in \mathbb{N}}\text{\textbackslash}\{b_{l_n}\}_{n\in \mathbb{N}\text{\textbackslash}\{s\}}$ and 
$\{\gamma_{k_n}\}_{n \in \mathbb{N}}$ determine the potential $q$ for any $s \in \mathbb{N}$, where $\sigma_{\alpha_2,\beta} = \{a_n\}_{n \in \mathbb{N}}$, $\sigma_{\alpha_1,\beta} = \{b_n\}_{n \in \mathbb{N}}$ are two spectra and 
$\{\gamma_n\}_{n \in \mathbb{N}}$ are 
point masses of the spectral measure $\mu_{\alpha_1,\alpha_2,\beta} = \sum_{n \in \mathbb{N}}\gamma_n \delta_{a_n}$.\\

(ii) If $\alpha_1 \neq 0$, $\alpha_2 \neq 0$ and there exists $N \in \mathbb{N}$ such that $k_n \neq l_n$ for all $n ~\textgreater~ N$, then 
$\{a_n\}_{n \in \mathbb{N}}$, 
$\{b_n\}_{n \in \mathbb{N}}\text{\textbackslash}\{b_{l_n}\}_{n\in \mathbb{N}\text{\textbackslash}\{s\}}$ and 
$\{\gamma_{k_n}\}_{n \in \mathbb{N}}$ determine the potential $q$ for any $s \in \mathbb{N}$, where 
$\sigma_{\alpha_2,\beta} = \{a_n\}_{n \in \mathbb{N}}$, $\sigma_{\alpha_1,\beta} = \{b_n\}_{n \in \mathbb{N}}$ are two spectra and 
$\{\gamma_n\}_{n \in \mathbb{N}}$ are 
point masses of the spectral measure $\mu_{\alpha_1,\alpha_2,\beta} = \sum_{n \in \mathbb{N}}\gamma_n \delta_{a_n}$.
\end{theorem}

\begin{proof}
In the proof of Theorem \ref{DDNDindex1} we used the inequalities (\ref{eigenvaluesBound}), namely
\begin{equation*}
 |\widetilde{b}_{l_m}(a_{k_n} - b_{l_m})(a_{k_n} - \widetilde{b}_{l_m})|^{-1} \leq 
 |\widetilde{b}_{k_n}(a_{k_n} - b_{k_n})(a_{k_n} - \widetilde{b}_{k_n})|^{-1} \leq 2/a_{k_n}^2.
\end{equation*}

If $\alpha_1 = 0$ or $\alpha_2 = 0$, these inequalities are still valid for any $m \in \mathbb{N}\text{\textbackslash}\{s\}$ and for 
sufficiently large $n \in \mathbb{N}$. Recall that $\widetilde{b}_{k_j} := b_{k_j}$ if ${k_j} \notin \{l_n\}_{n \in \mathbb{N}}$.

If $\alpha_1 \neq 0$, $\alpha_2 \neq 0$ and there exists $N \in \mathbb{N}$ such that $k_n \neq l_n$ for all $n ~\textgreater~ N$, we modify 
these inequalities as follows:
\begin{equation*}
 |\widetilde{b}_{l_m}(a_{k_n} - b_{l_m})(a_{k_n} - \widetilde{b}_{l_m})|^{-1} \leq 
 |\widetilde{b}_{{k_n+1}}(a_{k_n} - b_{{k_n+1}})(a_{k_n} - \widetilde{b}_{k_{n+1}})|^{-1} \leq 2/a_{k_n}^2,
\end{equation*}
which are valid for any $m \in \mathbb{N}\text{\textbackslash}\{s\}$ and for sufficiently large $n \in \mathbb{N}$.

After getting these inequalities we apply proofs of Lemma \ref{lmm2} and Theorem \ref{DDNDindex1} with the $m$-function 
$m_{\alpha_1,\alpha_2,\beta}$ and the spectral measure 
 $\mu_{\alpha_1,\alpha_2,\beta}$ and obtain uniqueness of $\{b_{l_n}\}_{n\in \mathbb{N}}$. Even though asymptotics of the 
 spectra may be different than Dirichlet-Dirichlet, Neumann-Dirichlet case, the same arguments can be used. After unique recovery of the 
 two spectra $\sigma_{\alpha_2,\beta}$ and $\sigma_{\alpha_1,\beta}$, Levinson's theorem uniquely determines the potential.\\
\end{proof}

\begin{theorem}\label{GBCindex2}
Let $q \in L^1(0,\pi)$, $\sin(\alpha_2 - \alpha_1) \neq 0$, $\alpha_1,\alpha_2,\beta \in [0,\pi)$ and 
$\prod_{n\in \mathbb{N}}a_{k_{n}}/b_{l_n}$ be absolutely convergent, where 
$\displaystyle\{a_{k_n}\}_{n\in \mathbb{N}} \subset \sigma_{\alpha_2,\beta}$, $\{b_{l_n}\}_{n\in \mathbb{N}} \subset \sigma_{\alpha_1,\beta}$.\\

(i) If $\alpha_1 = 0$ or $\alpha_2 = 0$, then $\{a_n\}_{n \in \mathbb{N}}$, 
$\{b_n\}_{n \in \mathbb{N}}\text{\textbackslash}\{b_{l_n}\}_{n\in \mathbb{N}}$ and 
$\{\gamma_{k_n}\}_{n \in \mathbb{N}}$ determine the potential $q$, where $\sigma_{\alpha_2,\beta} = \{a_n\}_{n \in \mathbb{N}}$, 
$\sigma_{\alpha_1,\beta} = \{b_n\}_{n \in \mathbb{N}}$ are two spectra and $\{\gamma_n\}_{n \in \mathbb{N}}$ are 
point masses of the spectral measure $\mu_{\alpha_1,\alpha_2,\beta} = \sum_{n \in \mathbb{N}}\gamma_n \delta_{a_n}$.\\

(ii) If $\alpha_1 \neq 0$, $\alpha_2 \neq 0$ and there exists $N \in \mathbb{N}$ such that $k_n \neq l_n$ for all $n ~\textgreater~ N$, then 
$\{a_n\}_{n \in \mathbb{N}}$, 
$\{b_n\}_{n \in \mathbb{N}}\text{\textbackslash}\{b_{l_n}\}_{n\in \mathbb{N}}$ and 
$\{\gamma_{k_n}\}_{n \in \mathbb{N}}$ determine the potential $q$, where 
$\sigma_{\alpha_2,\beta} = \{a_n\}_{n \in \mathbb{N}}$, $\sigma_{\alpha_1,\beta} = \{b_n\}_{n \in \mathbb{N}}$ are two spectra and 
$\{\gamma_n\}_{n \in \mathbb{N}}$ are 
point masses of the spectral measure $\mu_{\alpha_1,\alpha_2,\beta} = \sum_{n \in \mathbb{N}}\gamma_n \delta_{a_n}$.
\end{theorem}

\begin{proof}
If $\alpha_1 = 0$ or $\alpha_2 = 0$, we follow the proofs of Lemma \ref{lmm3} and Theorem \ref{DDNDindex2} with the $m$-function 
$m_{\alpha_1,\alpha_2,\beta}$ and the spectral measure $\mu_{\alpha_1,\alpha_2,\beta}$ and 
obtain uniqueness of $\{b_{l_n}\}_{n\in \mathbb{N}}$. After unique recovery of the two spectra 
$\sigma_{\alpha_2,\beta}$ and $\sigma_{\alpha_1,\beta}$, Levinson's theorem uniquely determines the potential.
 
 If $\alpha_1 \neq 0$, $\alpha_2 \neq 0$ and there exists $N \in \mathbb{N}$ such that $k_n \neq l_n$ for all $n ~\textgreater~ N$, then the 
 only difference appears in showing $\{1/(a_{k_n}-b_{l_n}\}_{n \in \mathbb{N}} \in l^1$, so let us show that absolute convergence of 
 $\prod_{n\in \mathbb{N}}(a_{k_n}/b_{l_n})$ implies $\{1/(a_{k_{n}} - b_{l_n})\}_{n \in \mathbb{N}} \in l^1$. Since 
 $\prod_{n\in \mathbb{N}}(a_{k_n}/b_{l_n})$ is absolutely convergent, 
 \begin{equation*}
  \sum_{n\in \mathbb{N}}\left|\frac{a_{k_{n}} - b_{l_n}}{b_{l_n}}\right| = 
  \sum_{n\in \mathbb{N}}\left|\frac{(k_n-1)^2 - (l_n - 1)^2 + \gamma_1 + \alpha_{k_n} - \beta_{l_n}}
  {(l_n - 1)^2 + \gamma_2 + (2/\pi)\int_0^{\pi}q(x)dx + \beta_{l_n}}\right| ~\textless~ \infty,
 \end{equation*}
 i.e. $\{(k_n^2 - l_n^2 - 2k_n + 2l_n)/l_n^2\}_{n \in \mathbb{N}} \in l^1$. Here $\gamma_1 = 2[\cot(\alpha_2)-\cot(\alpha_1)]/\pi$, 
 $\gamma_2 = 2[\cot(\beta)+\cot(\alpha_1)]/\pi$ and wlog we assume $\beta \neq 0$. Note that 
 $\lim_{n \rightarrow \infty} a_{k_n}/b_{l_n} = 1$ implies $\lim_{n \rightarrow \infty} k_n/l_n = 1$. Therefore 
 \begin{align*}
  \infty &~\textgreater~ \sum_{n\in \mathbb{N}}\left|\frac{k_n^2 - l_n^2 - 2(k_n - l_n)}{l_n^2}\right|\\ 
  &= \sum_{n\in \mathbb{N}}\frac{k_n + l_n -2}{l_n}\left|\frac{k_n - l_n}{l_n}\right|\\
  &\geq \sum_{n=1}^{N}\left|\frac{k_n - l_n}{l_n}\right| + \sum_{n=N+1}^{\infty}\frac{1}{l_n}\\
  &\geq c_1 \sum_{n\in \mathbb{N}}\frac{1}{l_n}
 \end{align*}
where $N \in \mathbb{N}$ and $c_1 ~\textgreater~ 0$, so $\{1/l_n\}_{n \in \mathbb{N}} \in l^1$ and hence by Limit Comparison Test 
$\{1/k_n\}_{n \in \mathbb{N}} \in l^1$. Therefore $\{1/(a_{k_{n}} - b_{l_n})\}_{n \in \mathbb{N}} \in l^1$, since for $n ~\textgreater~ N$, 
$1/|a_{k_{n}} - b_{l_n}| \leq 1/|a_{k_{n}} - b_{k_n+1}| = O(1/k_n)$ as $n$ goes to $\infty$. Now we apply proofs of Lemma \ref{lmm3} and 
Theorem \ref{DDNDindex2} with the $m$-function $m_{\alpha_1,\alpha_2,\beta}$ and the spectral measure $\mu_{\alpha_1,\alpha_2,\beta}$ and 
obtain uniqueness of $\{b_{l_n}\}_{n\in \mathbb{N}}$. After unique recovery of the two spectra 
$\sigma_{\alpha_2,\beta}$ and $\sigma_{\alpha_1,\beta}$, Levinson's theorem uniquely determines the potential.
 
\end{proof}

 \appendix
 \section{Complex function theoretical tools}
 In this section we recall some definitions and theorems from complex function theory used in our discussions. 
 We follow \cite{POL}.
 
 A function on $\mathbb{R}$ is Poisson-summable if it is summable with respect to the Poisson measure $\Pi$, defined as 
 $d\Pi := dx/(1+x^2)$. The space of Poisson-summable functions on $\mathbb{R}$ is denoted by $L_{\Pi}^1$.
 
 The Schwarz integral of a Poisson-summable function $f$ is
 \begin{equation*}
  Sf(z) = \frac{1}{i\pi}\int_{\mathbb{R}}\left(\frac{1}{t-z}-\frac{t}{1+t^2}\right)f(t)dt.
 \end{equation*}
The Schwarz integral of a real valued Poisson-summable function is given in terms of its Poisson and conjugate Poisson integrals:
\begin{align*}
 Sf &= Pf + iQf\\
 Pf(z) &= \frac{1}{\pi}\int\frac{y}{(t-x)^2 + y^2}f(t)dt\\
 Qf(z) &= \frac{1}{\pi}\int\left(\frac{x-t}{(x-t)^2+y^2} + \frac{1}{1+t^2}\right) f(t)dt
\end{align*}

A measure $\mu$ on $\mathbb{R}$ is Poisson-finite if $\int\frac{1}{1+t^2}d|\mu|(t) ~\textless~ \infty$. The Schwarz integral of a 
Poisson-finite measure $\mu$,  defined as 
\begin{equation*}
  S\mu(z) = \frac{1}{i\pi}\int_{\mathbb{R}}\left(\frac{1}{t-z}-\frac{t}{1+t^2}\right)d\mu(t),
 \end{equation*}
is analytic in the upper half-plane $\mathbb{C}_+$.

Outer functions in $\mathbb{C}_+$ are analytic functions of the form $e^{Sf}$ for $f \in L_{\Pi}^1$. 

Inner functions in $\mathbb{C}_+$ are bounded analytic functions with non-tangential boundary values, equal to $1$ in modulus, almost 
everywhere on $\mathbb{R}$. If an inner function extends to $\mathbb{C}$ meromorphically, it is called meromorphic inner function, 
usually denoted by $\Theta$.

Hilbert transform of $f \in L_{\Pi}^1$, denoted by $\widetilde{f}$, is defined as the singular integral
\begin{equation*}
 \widetilde{f}(x) = \frac{1}{\pi} ~p.v.\int\left[\frac{1}{x-t}+\frac{t}{1+t^2}\right]f(t)dt.
\end{equation*}
It is the angular limit of $Qf = \Im Sf$, hence the outer function $e^{Sf}$ coincides with $e^{f+i\widetilde{f}}$ on $\mathbb{R}$.

A meromorphic function is said to be real if it maps real numbers to real numbers on its domain. A meromorphic Herglotz function $m$ is a 
real meromorphic function with positive imaginary part on $\mathbb{C}_+$. It has negative imaginary 
part on $\mathbb{C}_-$ via the relation $m(\overline{z}) = \overline{m(z)}$. 

There is a one-to-one correspondence between meromorphic inner functions and meromorphic Herglotz functions via equations
\begin{equation*}
 m = i\frac{1+\Theta}{1-\Theta}, \quad \quad \quad \Theta = \frac{m-i}{m+i}.
\end{equation*}

A meromorphic Herglotz function can be described as the Schwarz integral of a positive discrete Poisson-finite measure:
\begin{equation*}
 m(z) = az + b + iS\mu,
\end{equation*}
where $a\geq 0$, $b \in \mathbb{R}$. The term $iS$ is also called the Herglotz integral and usually denoted by $H$. This representation is 
valid even if the Herglotz function can not be extended meromorphically to $\mathbb{C}$, in which case $\mu$ may not be discrete. 
It is called the Herglotz representation theorem. {\u C}ebotarev proved a similar result.
\begin{theorem} [{{\u C}ebotarev \cite{LEV}}] \label{Cebotarev} If the real meromorphic function $m$ maps $\mathbb{C}_+$ onto $\mathbb{C}_+$, then its poles 
$\{a_k\}_{k \in \mathbb{Z}}$ are all real and simple, and it may be represented in the form 
\begin{equation*}
 m(z) = az + b + \sum_{k=N}^{M} A_k\left(\frac{1}{a_k-z} - \frac{1}{a_k}\right),
\end{equation*}
where $a\geq 0$, $b \in \mathbb{R}$, $-\infty \leq N ~\textless~ M \leq \infty$, $A_k \geq 0$ and the sum 
$\sum_{k=N}^{M}A_k/a_k^2$ converges.  
\end{theorem}

\section*{Acknowledgement}
 I am grateful to Alexei Poltoratski for suggesting the problem of this paper and for all the help and support he has given me.

\end{document}